\documentclass[a4paper,11pt]{article}
\usepackage{}
\usepackage{bbm}
\usepackage{CJK}
\usepackage{latexsym}
\usepackage{amssymb}
\usepackage{amssymb,amsmath,amsthm}
\usepackage{tikz-cd}
\usepackage[numbers,sort&compress]{natbib}
\usepackage[colorlinks,linkcolor=blue,anchorcolor=blue,citecolor=blue]{hyperref}
\usepackage[colorinlistoftodos]{todonotes}

\renewcommand{\thefootnote}{\fnsymbol}

\newtheorem{theorem}{Theorem}[section]

\newtheorem{corollary}{Corollary}[section]
\newtheorem{lemma}{Lemma}[section]

\newtheorem{definition}{Definition}[section]
\newtheorem{remark}{Remark}[section]
\newtheorem{proposition}{Proposition}[section]

\large\normalsize

\date{}

\begin{document}

\title{Structure of some mapping spaces  %
\footnotemark[0]}

\author{Liangzhao Zhang \footnotemark[0]\qquad  Xiangyu Zhou \footnotemark[0]\qquad}
\renewcommand{\thefootnote}{\fnsymbol{footnote}}
\footnotetext[0]{email addresses:zhangliangzhao@amss.ac.cn(Liangzhao Zhang),xyzhou@math.ac.cn(Xiangyu Zhou)}
\maketitle

\begin{abstract}
	We prove that the path space of a differentiable manifold is diffeomorphic to a Fréchet space, endowing the path space with a linear structure. Furthermore, the base point preserving mapping space consisting of maps from a cube to a differentiable manifold is also diffeomorphic to a Fréchet space. As a corollary of a more general theorem, we establish that the path fibration becomes a fibre bundle for manifolds M. Additionally, we discuss the mapping space from a compact topological space to a differentiable manifold, demonstrating that this space admits the structure of a smooth Banach manifold.
\end{abstract}


\subsection{ Research Contributions and Results}

Path spaces, defined as collections of differentiable paths on a manifold, play an important role in understanding the geometric and topological properties of manifolds.

In this article, we establish several fundamental results regarding the structure of path spaces and mapping spaces on manifolds.

1. Path Spaces as Banach Spaces: We prove that for any manifold $ M $, the path space $ C^k_p(I, M)$, consisting of all $ C^k $ paths starting at $ p \in M $, is diffeomorphic to the Banach space $ C^{k-1}(I, T_p M) $. The diffeomorphism is explicitly defined in Chapter 1, and the proof is presented in Chapter 3. This result endows the path space with a linear structure, which may be useful in some related fields.

2. Higher-Dimensional Mapping Spaces: We extend this result to higher-dimensional mapping spaces, showing that $ C^\infty_p([-1, 1]^n, M) $ is diffeomorphic to $ \oplus_{i=1}^n C^\infty([-1, 1]^i, T_p M) $. 
Furthermore, for mixed regularity conditions, we prove that
 $$ C^{(k_1, \dots, k_n)}_p([-1, 1]^n, M) \simeq\oplus_{i=1}^n C^{(k_1, \dots, k_i-1)}([-1, 1]^i, T_p M) $$ for $ k_i \geq 1 $. These results generalize the linear structure of path spaces to some mapping spaces.
As a corollary, we show that the compact-open topology of $ C^\infty_p(\mathbb{R}^n, M) $ is homeomorphic to the Fréchet space $ \oplus_{i=1}^n C^\infty(\mathbb{R}^i, T_p M) $. 

3. Banach and Fréchet Manifold Theory: In Chapter 2, we develop the theory of Banach (or Fréchet) manifolds consisting of $ C^{(k_1, \dots, k_n)} $ maps. 

4. Weierstrass Approximation Theorem for curves: In Chapter 3, we introduce the concept of polynomial-like curves and prove a Weierstrass approximation theorem, demonstrating that polynomial-like curves are dense in $ C^k(I, M) $.  

5. Fiber Bundle Structures: In Chapter 4, we prove that for any cofibration $ i: A \hookrightarrow X $ between topological spaces, if $ A $ is compact and $ M $ is a manifold, then the map $ i^*: C(X, M) \to C(A, M) $ is a fiber bundle. We also establish a smooth version of this result. As a corollary, we show that the path fibration $ \pi: C^k_p(I, M) \to M $, defined by $ \gamma \mapsto \gamma(0) $, is a smooth fiber bundle when $ M $ is a manifold. This implies that for any manifold $ M $, there exists a fiber bundle $ E $ on $ M $ which diffeomorphic to the Banach space $ C^k(I, \mathbb{R}^{\dim M}) $.

6. Shape Spaces as Vector Bundles: We further prove that the shape space $ \text{Imm}_p^k(\mathbb{R}, M) / \text{Diff}_+(\mathbb{R}) $ is an infinite-rank vector bundle on $ S^{\dim M - 1} $. 

7. Banach Manifold Structure for Continuous Mappings: In Chapter 5, we demonstrate that the space of continuous mappings $ C(A, M) $ is a Banach manifold when $ A $ is a compact topological space and $ M $ is a manifold. Additionally, if $ M $ is a complex manifold admitting a holomorphic connection, then $ C(A, M) $ is a complex Banach manifold for any compact topological space $ A $.

\section{The diffeomorphism $P$}
Let $M$ be a manifold and choose a complete Riemann metric $g$ on it. Suppose  $\nabla$ is a connection compatible with the metric $g$, and it doesn't have to be the Levi-Civita connection. 
$P^{a\rightarrow b}_{\gamma}$ be the parallel transport along a curve $\gamma$ from $\gamma(a)$ to $ \gamma (b)$, and $k\geq 1$, then we define 
\begin{equation}
	\begin{array}{ccc}
		P:C_{p}^{k}(I,M) & \longrightarrow & C^{k-1}(I,T_{p}M)\\
		t\mapsto\text{\ensuremath{\gamma(t)}} & \longmapsto & t\mapsto P_{\gamma}^{t\rightarrow0}(\dot{\gamma}(t))
	\end{array}
\end{equation}
Here $C_{p}^{k}(I,M)$ is the space of all $C^k$ curves starting at $p\in M$, and $I$ is the interval $[0,1]$.

First we show that P is an bijection, by constructing $P^{-1}$ as follows.

We derive the equation that $P^{-1}(v)$ satisfies.
Choose a basis $\{e_i\}$ of $T_pM$, for any $\gamma\in C_{p}^{k}(I,M) $ ,let $\{e_i\}$ parallel transport along $\gamma$. Then we get a frame $\{e_i(t)\}$ which is parallel along $\gamma$. 
If $v(t)=P^{t\rightarrow0}_\gamma(\dot{\gamma}(t))$, then$\dot{\gamma}(t)=P^{0\rightarrow t}_\gamma(v(t))$. For
$v(t)=v^i(t)e_i$, we have
\begin{equation}\label{gamma and v(t)}
	\dot{\gamma}(t)=P^{0\rightarrow t}_\gamma(v(t))=v^i(t)e_i(t)
\end{equation}
So we have the equation about the curve $\gamma$ and the frame $f_i$ along $\gamma$:
\begin{equation}\label{global equation for P^-1}
	\begin{cases}
		\nabla_{\dot{\gamma}}f_{i}(t) =0\\
		\dot{\gamma}(t) =v^{i}(t)f_{i}(t)\\
		\gamma(0)=p,\   f_{i}(0)=e_{i}(0)
	\end{cases}
\end{equation}
Choose a coordinate chart $(U_0,\varphi_0)$ near $\gamma(0)=p$.

Assume $\varphi_{0}(\gamma(0))=0$, and $\{\frac{\partial}{\partial x_i}\}$ is the natural frame. Let
$$\dot{ \gamma}(t)=r^i(t)\frac{\partial}{\partial x_i}(\gamma(t)),\ e_i(t)=e_i^j(t)\frac{\partial}{\partial x_j}(\gamma(t))$$
So in local coordinate, the equation $\nabla_{\dot{\gamma}}e_i(t)=0$ can be written as:
\begin{equation}
	\begin{split}
		0&=\nabla_{\dot{\gamma}}e_i(t)=\nabla_{\dot{\gamma}}e_i^j(t)\frac{\partial}{\partial x_j}\\
		&=\frac{de_{i}^{l}(t)}{dt}\frac{\partial}{\partial x_{l}}+e_{i}^{j}(t)\varGamma_{kj}^{l}(\gamma(t))r^{k}(t)\frac{\partial}{\partial x_{l}}
	\end{split}
\end{equation}
so we have 
\begin{equation}\label{parallel frame}
	\frac{de_{i}^{l}(t)}{dt}=-e_{i}^{j}(t)\varGamma_{kj}^{l}(\gamma(t))r^{k}(t)	
\end{equation}
for all $i,l\in \{1,2,\dots,m\}$.

If we combine (\ref{parallel frame}) with (\ref{gamma and v(t)}) ,we get an ordinary differential equation system which has $m^2+m$ variables: 
\begin{equation}\label{local equation}
	\begin{cases}
		\dfrac{de_{i}^{l}(t)}{dt}=-e_{i}^{j}(t)\varGamma_{kj}^{l}(\gamma(t))r^{k}(t) \\
		r^{k}(t)=v^{i}(t)e_{i}^{k}(t)\\
		$initial value:$\varphi _0(\gamma(0))=0,e_i^j(0)=e_i^j
	\end{cases}
\end{equation}
For $\Gamma^k_{ij}$ is $C^{\infty}$, $v(t)$ is continuous, we know that the equation satisfies the local Lipschitz condition, and there exists a unique solution locally.

\begin{proposition}\label{The equation has the following 3 properties}
	The equation has the following 3 properties:

	(1).$\gamma(t)$ is independent of the choice of basis $\{e_i\}$ in $T_p M$
	
	(2).$\gamma(t)$ is independent of the choice of chart $(U_0,\varphi_0)$ around $\gamma(0)$
	
	(3).If $(M,g)$ is complete, and $v(t)$ is defined on $\mathbb{R}$, then the equation always has a global solution $\gamma :\mathbb{R}\longrightarrow(M,g)$.
\end{proposition}
\begin{proof}

First we prove 1. Assume there is another basis $\{ \tilde{e}_i(0)\}$ of $T_p M$

Because $v(t)=v^i(t)e_i(0)=\tilde{v}^{i}(t)\tilde{e}_{i}(0)$, so  $\exists\left[a_{i}^{j}\right]\in GL(m,\mathbb{R}),$ such that$\tilde{v}^{i}(t)=v^{j}(t)a_{j}^{i},a_{i}^{j}\tilde{e}_{j}^{k}(t)=e_{i}^{k}(t)$.
Now we can show that the solution $\gamma (t)$ to these two equation are equal.
\begin{equation}
	\begin{cases}
		\dfrac{df_{i}^{l}(t)}{dt}=-f_{i}^{j}(t)\varGamma_{kj}^{l}(\gamma(t))r^{k}(t)\\
		r^{k}(t)=v^{i}(t)f_{i}^{k}(0)\\
		\varphi_{0}(\gamma(0))=0,f_{i}^{j}(0)=e_{i}^{j}(0)
	\end{cases}
\end{equation}

\begin{equation}\label{equation for f}
	\begin{cases}
		\dfrac{d\tilde{f}_{i}^{l}(t)}{dt}=-\tilde{f}_{i}^{j}(t)\varGamma_{kj}^{l}(\tilde{\gamma}(t))\tilde{r}^{k}(t)\\
		\tilde{r}^{k}(t)=\tilde{v}^{i}(t)\tilde{f}_{i}^{k}(0)\\
		\varphi_{0}(\tilde{\gamma}(0))=0,\tilde{f}_{i}^{j}(0)=\tilde{e}_{i}^{j}(0)
	\end{cases}
\end{equation}
Multiply both sides of (\ref{equation for f}) by $\left[a_{i}^{j}\right]$, and use $\tilde{v}^{i}(t)=v^{j}(t)a_{j}^{i}$, 
Then we have:
\begin{equation}
	\begin{cases}
		\dfrac{d\left(a_{i}^{s}\tilde{f}_{s}^{l}(t)\right)}{dt}=-a_{i}^{s}\tilde{f}_{s}^{j}(t)\varGamma_{kj}^{l}(\tilde{\gamma}(t))r^{k}(t)\\
		\tilde{r}^{k}(t)=v^{j}(t)a_{j}^{i}\tilde{f}_{i}^{k}(0)\\
		\varphi_{0}(\tilde{\gamma}(0))=0,a_{i}^{s}\tilde{f}_{s}^{j}(0)=a_{i}^{s}\tilde{e}_{s}^{j}(0)=e_{i}^{j}(0)
	\end{cases}
\end{equation}
So $(a_{i}^{s}\tilde{f}_{s}^{j},\tilde{\gamma}(t))$ and $(f_{i}^{j},\gamma(t))$ satisfy the same initial value problem, hence
$\tilde{\gamma}(t)=\gamma(t)$, and $\gamma(t)$ is independent of the choice of basis $e_i(0)$.

Next we prove 2.

For the equation (\ref{local equation}) can be write globally as in \ref{global equation for P^-1}, It's obvious that the solution does not depend on the choice of local coordinates.

Finally, we prove 3.

It's obvious that the maximal existence interval of $\gamma$ can't be closed interval. Because if $\gamma$ can be defined on $[0,T]$, then we can use initial value $\gamma(T),e_i^j(T)$,  to get a solution on $[T,T+\epsilon]$, so the maximal existence interval of $\gamma$ can't be closed interval.

Suppose the maximal existence interval of $\gamma$ is $[0,T),T< +\infty$. For
$P_{\gamma }^{t\rightarrow 0}(\dot{\gamma}(t))=v(t)$ is bounded on $[0,T]$, and the connection is compatible with metric, we have $||\dot{\gamma}(t)||$ is bounded on $[0,T]$. So $\gamma(t)$ has at most one limit point when $t\rightarrow T$ and the sequence ${\gamma(T-\frac{1}{n})}$ is a Cauchy sequence. Then by the completeness of $(M,g)$, we know that ${\gamma(T-\frac{1}{n})}$ has a limit point,so $\exists q\in M,\ \underset{t\rightarrow T}{\lim}\gamma(t)=q$ . 

In (\ref{parallel frame}), $\Gamma_{kj}^l(\gamma(t))$ and $r^k(t)$ are continuous on $[0,T]$, hence are bounded. For the connection is compatible with metric, $||e_i(t)||$ is constant, so $e^j_i(t)$ is bounded on $[0,T)$. Then by (\ref{parallel frame}), $\frac{de_{i}^{l}(t)}{dt}$ is bounded on $[0,T)$. By the completeness of $TM$ and a similar discussion as above, we know that $\underset{t\rightarrow T}{\lim}e_{i}(t)=e_{i}(T)$ exists. 
So we can use $q, e_i(T)$ as initial value to get a solution on $[T,T+\varepsilon)$, contradict with the maximality of $[0,T)$. So $\gamma $ can be defined on $[0,+\infty)$. And the equation $P^{t\rightarrow 0}_{\alpha}\left( \frac{\partial}{\partial t}(\alpha(-t))\right) =-v(-t)$ has a solution $\alpha(-t)$ defined for $t\in [0,+\infty)$. Then we let $\gamma(t)=\alpha(t)$ for $v\leq0 $, and we have an global solution $\gamma :\mathbb{R}\longrightarrow(M,g)$.
\end{proof}
So $P^{-1}$ is well defined, and $P$ is a bijection.  We will prove $P$ is a diffeomorphism in chapter 3.

\section{Preliminaries}
	\subsection{Ordinary differential equations and parallel transportation}
	\begin{definition}
		Let $U $ be an open subset of $[0,\infty)^{n-b}\times \mathbb{R}^b\subseteq \mathbb{R}^n$, 
		then we define 
		$$C^{(k_1,\cdots,k_n)}(U,V):=\{f:U\longrightarrow V\mid \forall m_{i}\leq k_{i},\dfrac{\partial^{m_{1}}}{\partial x_{1}^{m_{1}}}\cdots\dfrac{\partial^{m_{n}}}{\partial x_{n}^{m_{n}}}f(x_{1,}\cdots x_{n}) \ is \ continuous\} $$
		Or if we have a direct sum,  
		$U=\Pi_{i=1}^{n} U_i$ and $U_i$ is an open subset of $[0,\infty)^{m_i-b_i}\times \mathbb{R}^{b_i}\subseteq \mathbb{R}^{m_i}$, then we define 
		$$C^{(k_1,\cdots,k_n)}(\Pi_{i=1}^{n} U_i,V):=\{f:\Pi_{i=1}^{n} U_i\longrightarrow V\mid \forall |m_{i}|\leq k_{i},\dfrac{\partial^{m_{1}}}{\partial x_{1}^{m_{1}}}\cdots\dfrac{\partial^{m_{n}}}{\partial x_{n}^{m_{n}}}f(x_{1,}\cdots x_{n}) \ is \ continuous\} $$ where $m_i$ are muti-indexes, and $|m_{i}|:=|(m_{i1},\cdots,m_{in_i})|=\underset{j}{\sum}m_{ij}$.
		Or sometimes, we say a function $f$ has $(k_1,\cdots,k_n)$-th order continuous partial derivatives, if it is $C^{(k_1,\cdots,k_n)}$.
	\end{definition}
	\begin{lemma}[smoothness of solution to ordinary differential equations.]\label{smoothness of solution to ode}
		Given an arbitrary ordinary differential equation:
		\begin{equation}\label{general ode}
			\begin{cases}
				\dfrac{dy_{i}}{dt}(t,\lambda_1,\cdots, \lambda_{n})=f(t,y_{1},\cdots,y_{l},\lambda_{1},\cdots,\lambda_{n})\\
				y(t_{0}(\lambda_{1},\cdots,\lambda_{n}),\lambda_{1},\cdots,\lambda_{n})=y_{0}(\lambda_{1},\cdots,\lambda_{n})
			\end{cases}
		\end{equation}
		Suppose $k_i\geq 0$, $f(t,y,\lambda_{1},\cdots,\lambda_{n})$ is $C^{(k_0,\infty,k_1,\cdots,k_n)}$, and $t_{0}(\lambda_{1},\cdots,\lambda_{n})$, $y_{0}(\lambda_{1},\cdots,\lambda_{n})$ are $C^{(k_1,\cdots,k_n)}$.
		
		Then the solution $y=y(t,\lambda_{1},\cdots,\lambda_{n})$ of equation (\ref{general ode}) is $C^{(k_{0}+1,k_1,\cdots,k_n)}$.
	\end{lemma}
	\begin{proof}
		We prove by induction.
		
		When $\sum_{i=0}^{n}k_{i}=0 $, it is just the continuously dependency of solutions on parameters.
		
		Suppose we have proved this lemma for $\sum_{i=0}^{n}k_{i}=q $, and we are going to prove it for$\sum_{i=0}^{n}k_{i}=q+1$.
		
		If $k_1,\cdots,k_n$ are not all zero, then suppose $k_{i}\neq 0$.
		Take the derivative of both sides of equation (\ref{general ode}) with respect to $\lambda_i$. By continuously differentiablity of solution to ordinary differential equations, $\dfrac{\partial y}{\partial\lambda_{i}}$ is continuous. Because $\dfrac{\partial y}{\partial t}$ and $ \dfrac{\partial}{\partial\lambda_{i}}\dfrac{\partial y}{\partial t}$ are continuous, we have $\dfrac{\partial}{\partial\lambda_{i}}\dfrac{\partial y}{\partial t}=\dfrac{\partial}{\partial t}\dfrac{\partial y}{\partial\lambda_{i}}$. Then we have:
		\begin{equation}\label{partial lambda of both sides of equation}
			\begin{cases}
				\dfrac{\partial}{\partial t}\dfrac{\partial y}{\partial\lambda_{i}}(t,\lambda) & =\dfrac{\partial f}{\partial\lambda_{i}}(t,y(t,\lambda),\lambda)+\dfrac{\partial f}{\partial y}(t,y(t,\lambda),\lambda)\dfrac{\partial y}{\partial\lambda_{i}}(t,\lambda)\\
				\dfrac{\partial y}{\partial\lambda_{i}}(t_{0},\lambda) & =\dfrac{\partial y_{0}}{\partial\lambda_{i}}(\lambda)-\dfrac{\partial y}{\partial t}(t_{0},\lambda)\dfrac{\partial t_{0}}{\partial\lambda_{i}}(\lambda)
			\end{cases}
		\end{equation}
		Let $\tilde{f}(t,\tilde{y},\lambda):=\dfrac{\partial f}{\partial\lambda_{i}}(t,y,\lambda)+\dfrac{\partial f}{\partial y}(t,y,\lambda)\tilde{y}$, then $\tilde{f}$ is $C^{(k_0,\infty,k_1,\cdots,k_{i}-1,\cdots,k_n)}$. And by induction, the lemma holds for $\sum_{i=0}^{n}k_{i}=q $, so $\dfrac{\partial y}{\partial t}$ is $C^{(k_0,k_1,\cdots,k_{i}-1,\cdots,k_n)}$, hence the initial value of equation \ref{partial lambda of both sides of equation} is $C^{(k_0,k_1,\cdots,k_{i}-1,\cdots,k_n)}$. Then $\dfrac{\partial y}{\partial\lambda_{i}}(t_0,\lambda)$ is $C^{(k_0+1,k_1,\cdots,k_{i}-1,\cdots,k_n)}$.
		
		If $k_1=\cdots=k_n =0$, then $k_{0}=q+1$. Take the derivative of both sides of equation (\ref{general ode}) with respect to $t$:
		\begin{equation}
			\begin{cases}
				\dfrac{\partial}{\partial t}\dfrac{\partial y}{\partial t}(t,\lambda) & =\dfrac{\partial f}{\partial t}(t,y(t,\lambda),\lambda)+\dfrac{\partial f}{\partial y}(t,y(t,\lambda),\lambda)\dfrac{\partial y}{\partial t}(t,\lambda)\\
				\dfrac{\partial y}{\partial t}(t_{0}(\lambda),\lambda) & =f(t_{0}(\lambda),y_{0}(\lambda),\lambda)
			\end{cases}
		\end{equation}
		Then $\dfrac{\partial y}{\partial t}(t,\lambda)$ has $(q,0,\cdots,0)$-th order continuous partial derivatives by induction, which completes the proof. 
	\end{proof}

	\begin{lemma}[Parallel transport of smooth vector fields]\label{Parallel transport of smooth vector fields}
		Let
		\begin{equation}
			\begin{array}{cccc}
				\alpha: & I^{n+1} & \longrightarrow & M\\
				& (s_{1},\cdots,s_{n},t) & \longmapsto & \alpha(s_{1},\cdots,s_{n},t)
			\end{array}
		\end{equation}
		be a $C^{(k_1,\cdots,k_{n+1})}$ map, $k_i\geq 0$ and $\dfrac{\partial \alpha}{\partial t}$ is continuous. 
		$X$ is a $C^{(k_1,\cdots,k_{n+1})}$ vector field along $\alpha$. 
		Then $P^{t\rightarrow 0}_{\alpha_{s}}X(s,t)$ is a $C^{(k_1,\cdots,k_{n+1})}$ vector field along $\beta:(s_{1},\cdots,s_{n},t)\longmapsto\alpha(s_{1},\cdots,s_{n},0)$.
	\end{lemma}
	\begin{proof}

		Let $\tilde{\alpha}:(s,t,u)\longrightarrow P_{\alpha_{s}}^{t\rightarrow u}X(s,t)$.
		Then in local coordinate system $(U,\varphi)$, $\tilde{\alpha}$ satisfies the following equation:
		\begin{equation}
			\begin{cases}
				\dfrac{d}{du}\tilde{\alpha}^{l}(s,t,u)=-\dfrac{\partial\alpha^{i}}{\partial t}(s,u)\tilde{\alpha}^{j}(s,t,u)\Gamma_{ij}^{l}(\alpha(s,u))=:\tilde{f}(s,t,u,\tilde{\alpha})\\
				\tilde{\alpha}(s,t,t)=X(s,t)
			\end{cases}
		\end{equation}
		Let $\hat{\alpha}:(s,t,u)\longmapsto\alpha(s,u)$, Then $\tilde{f}$ is defined on $\hat{\alpha}^{-1}(U)\times \mathbb{R}^m$
		
		It's clear that $\tilde{f}(s,t,u,\tilde{\alpha})$ is $C^{(k_1,\cdots,k_n,\infty ,k_{n+1}-1,\infty)}$, and $X(s,t)$ is $C^{(k_1,\cdots,k_{n+1})}$. By lemma \ref{smoothness of solution to ode}, we know that $\tilde{\alpha}(s,t,u)$ is 
		$C^{(k_1,\cdots,k_n,k_{n+1},k_{n+1})}$(on the region where the solution exists).
		
		To prove $\tilde{\alpha}$ has enough smoothness on $I^{n+2}$ , we need a lemma:
		\begin{lemma}\label{lemma lebesgue number}
			Let $f:I^{n} \longrightarrow M$ be continuous map, then $\exists \varepsilon >0$, for all subset $X_{\varepsilon}\subseteq I^{n}$ with diameter $diam(X_{\varepsilon})<\varepsilon$, $f(X_{\varepsilon})$ is contained in some coordinate domain $(U,\varphi)$ of M. 
		\end{lemma}
		\begin{proof}
			Let $\{U_i\}$ be a coordinate cover of M, then $\{f^{-1}(U_i)\}$ is an open cover of $I^{n}$. Suppose $\lambda$ is the Lebesgue number of covering $\{f^{-1}(U_i)\}$, then we can let $\varepsilon < \lambda$, which completes the proof. 
		\end{proof}
		
		For $\hat{\alpha}(s,t,u)=\alpha(s,u)$ is continuous, so $\exists \varepsilon>0$, such that each cube $X_{\varepsilon}\subseteq I^{n+2}$ with side length $\varepsilon$ satisfies $\hat{\alpha}(X_{\varepsilon})\subseteq U_i$ for some coordinate neighborhood $U_i$ of $M$.
		
		So if the initial value of $\tilde{\alpha}(s,t,u)$ in $X_{\varepsilon}$  is $C^{(k_1,\cdots,k_n,k_{n+1})}$, then $\tilde{\alpha}(s,t,u)$ is $C^{(k_1,\cdots,k_n,k_{n+1},k_{n+1})}$ in $X_{\varepsilon}$. For $\tilde{\alpha}(s,t,t)=X(s,t)$ is $C^{(k_1,\cdots,k_{n+1})}$, we conclude that 
		$\tilde{\alpha}(s,t,u)$ is $C^{(k_1,\cdots,k_n,k_{n+1},k_{n+1})}$ in $I^{n+2}$. 
		
		So the vector field $P^{t\rightarrow 0}_{\alpha_{s}}X(s,t)=\tilde{\alpha}(s,t,0)$ is $C^{(k_1,\cdots,k_{n+1})}$. 
	\end{proof}
	
	\begin{lemma}\label{exponential law for any topo space}
		Let $X,Y,Z$ be topological spaces, then $f\in C(X\times Y,Z)\Longrightarrow f^{\lor }\in C(X,C(Y,Z))$, where $f^{\lor }:x\longmapsto (y\mapsto f(x,y))$
	\end{lemma}
	\begin{proof}
		If $f\in C(X\times Y,Z)$, Let $U^K:=\{g:Y\longrightarrow Z|g(K\subseteq U)\}$. Suppose $f^{\lor}(x_0)\in U^K$, then $f(\{x_0\}\times K)\subseteq U$. So $\forall (x_0,y)\in \{x_0\}\times K$, $\exists U_{1,y}\times U_{2,y}\subseteq f^{-1}(U)$ such that $(x_0,y)\in U_{1,y}\times U_{2,y}$. Then $\{U_{1,y}\times U_{2,y}\}_{y\in K}$ is an open cover of $\{x_0\}\times K$, and there exists a finite subcover $\{U_{1,y_i}\times U_{2,y_i}\}_{1\leq i\leq n}$. Let $U_{x_0}:=\bigcap_{i=1}^{n}U_{1,y_i}$, then $f(U_{x_0}\times K)\subseteq U$, hence $U_{x_0}\subseteq (f^{\lor})^{-1}(U^K)$.
	\end{proof}
	\begin{lemma}\label{exponential law for Y locally compact}
		Let $X,Y,Z$ be topological spaces, and $Y$ is locally compact then $f\in C(X,C(Y,Z))\Longrightarrow f^\wedge\in C(X\times Y,Z)$, where $f^{\wedge }:(x,y)\longmapsto f(x)(y)$
	\end{lemma}
	\begin{proof}
		If $f\in C(X,C(Y,Z))$, let $(x_0,y_0)\in (f^\wedge)^{-1}(U)$, then $f^{\wedge}(x_0,y_0)\in U$. Take a compact neighborhood $K_{y_0}$ of $y_0$, then $U^{K_{y_0}}$ is open in $C(Y,Z)$, hence $f^{-1}(U^{K_{y_0}})$ is open in $X$. finally, take a open neighborhood $U_{y_0}$ of $y_0$ in $K_{y_0}$, then $(x_0,y_0)\in f^{-1}(U^{K_{y_0}})\times U_{y_0}\subseteq (f^\wedge)^{-1}(U)$, so $f^{\wedge}$ is continuous.
	\end{proof}

	\subsection{Preliminaries about Banach manifolds}
	In this section, we set $k_i\in\mathbb{Z}_{\geq 0}$.
	\begin{definition}
		Let U be an open set in $[0,\infty)^{n-b}\times \mathbb{R}^b$, $k_i\geq 0$ and $V$ is a (finite dimensional) vector space. The 
		$C^{(k_{1},\cdots,k_{n})}$ compact open topology on $C^{(k_{1},\cdots,k_{n})}(U,V)$ is defined to be the initial topology with respect to the family of mappings
		\begin{equation}
			\begin{array}{cccc}
				\partial^{(l_{1},\cdots,l_{n})}: & C^{(k_{1},\cdots,k_{n})}(U,V) & \longrightarrow & C(U,V)\\
				& g & \longmapsto & \partial_{u}^{(l_{1},\cdots,l_{n})}g:=\frac{\partial^{l_{1}}}{\partial u_{1}^{l_{1}}}\cdots\frac{\partial^{l_{n}}}{\partial u_{n}^{l_{n}}}g
			\end{array}
		\end{equation}
		which consists every $\partial^{(l_{1},\cdots,l_{n})}$ such that $l_i\leq k_i\in\mathbb{Z}_{\geq 0}$.
	\end{definition}
\begin{theorem}\label{general exponential law}
	Let $U$ be an open set in $[0,\infty)^{n-1-b}\times \mathbb{R}^b$, or $U=I^{n-1}$, then $f\in C^{(k_{1},\cdots,k_{n})}(\mathbb{R}\times U,V)\Longleftrightarrow f^{\lor}\in C^{k_{1}}\left(\mathbb{R},C^{(k_{2},\cdots,k_{n})}(U,V)\right)$, here 
	\begin{equation}
		\begin{array}{cccc}
			f^{\lor}: & \mathbb{R} & \longrightarrow & C^{(k_{2},\cdots,k_{n})}(U,V)\\
			& r & \longmapsto & u\mapsto f(r,u)
		\end{array}
	\end{equation}
\end{theorem}
\begin{proof}
	First, we prove the case $k_1=0$. By definition of the topology on $C^{(k_{2},\cdots,k_{n})}(U,V)$, $f^\lor\in C^{0}\left(\mathbb{R},C^{(k_{2},\cdots,k_{n})}(U,V)\right)\Longleftrightarrow \partial^{(l_{2},\cdots,l_{n})}\circ f^{\lor}$ is continuous for any $l\leq k$ (that means $l_i \leq k_i$). Here $\partial^{(l_{2},\cdots,l_{n})}$ is the map:
	\begin{equation}\label{partial operator}
		\begin{array}{cccc}
			\partial^{(l_{2},\cdots,l_{n})}: & C^{(k_{2},\cdots,k_{n})}(U,V) & \longrightarrow & C^{(k_{2}-l_{2},\cdots,k_{n}-l_{n})}(U,V)\\
			& g & \longmapsto & \partial_{u}^{(l_{2},\cdots,l_{n})}g:=\frac{\partial^{l_{2}}}{\partial u_{1}^{l_{2}}}\cdots\frac{\partial^{l_{n}}}{\partial u_{n-1}^{l_{n}}}g
		\end{array}
	\end{equation}
	Notice that $\left(\partial_{u}^{(l_{2},\cdots,l_{n})}f\right)^{\lor}=\partial^{(l_{2},\cdots,l_{n})}\circ f^{\lor}$ and we have:
	\begin{equation}
		\begin{array}{ccc}
			f\in C^{(0,k_{2},\cdots,k_{n})}(\mathbb{R}\times U,V) & \Longleftrightarrow & \partial_{u}^{(l_{2},\cdots,l_{n})}f\in C^{0}(\mathbb{R}\times U,V),\forall l\leq k\\
			& \Longleftrightarrow & \left(\partial_{u}^{(l_{2},\cdots,l_{n})}f\right)^{\lor}\in C^{0}(\mathbb{R},C^{0}(U,V),\forall l\leq k\\
			& \Longleftrightarrow & \partial^{(l_{2},\cdots,l_{n})}\circ f^{\lor}\in C^{0}(\mathbb{R},C^{0}(U,V),\forall l\leq k
		\end{array}
	\end{equation}
	the second line is resulted from the exponential law for continuous maps, and the third line is because $\left(\partial_{u}^{(l_{2},\cdots,l_{n})}f\right)^{\lor}=\partial^{(l_{2},\cdots,l_{n})}\circ f^{\lor}$.
	
	If $k\geq 1$, we claim that if $f\in C^{(k_{1},k_{2},\cdots,k_{n})}(\mathbb{R}\times U,V)$ or $f^\lor\in C^{k_1}(\mathbb{R},C^{(k_{2},\cdots,k_{n})}(U,V))$, then $\dfrac{df^{\lor}}{dr}=\left(\dfrac{\partial f}{\partial r}\right)^{\lor}$. Here $\dfrac{df^{\lor}}{dr}$ is the derivation of the curve $f^\lor :\mathbb{R}\longrightarrow C^{(k_{2},\cdots,k_{n})}(U,V)$. Suppose we proved the claim, and we assume that we have done for $k_1\leq m$. Let $k_1=m+1$, and let $f\in C^{(m+1,k_{2},\cdots,k_{n})}(\mathbb{R}\times U,V)$, then $\dfrac{\partial f}{\partial r}\in C^{(m,k_{2},\cdots,k_{n})}(\mathbb{R}\times U,V)$. So $\dfrac{df^{\lor}}{dr}=\left(\dfrac{\partial f}{\partial r}\right)^{\lor}\in C^{m}(C^{(k_{2},\cdots,k_{n})}(U,V))$ by induction and we are done.
	
	On the other hand, if $f^\lor\in C^{m+1}(\mathbb{R},C^{(k_{2},\cdots,k_{n})}(U,V))$, then by $\left(\dfrac{\partial f}{\partial r}\right)^{\lor}=\dfrac{df^{\lor}}{dr}\in C^{m}(\mathbb{R},C^{(k_{2},\cdots,k_{n})}(U,V))$, we know that $f\in C^{(m+1,k_{2},\cdots,k_{n})}(\mathbb{R}\times U,V)$.
	
	So now we only need to prove the claim $\dfrac{df^{\lor}}{dr}=\left(\dfrac{\partial f}{\partial r}\right)^{\lor}$, if $f\in C^{(k_{1},k_{2},\cdots,k_{n})}(\mathbb{R}\times U,V)$ or $f^\lor\in C^{k_1}(\mathbb{R},C^{(k_{2},\cdots,k_{n})}(U,V))$.
	If $f^\lor\in C^{k_1}(\mathbb{R},C^{(k_{2},\cdots,k_{n})}(U,V))$, the proof is trivial, the only trouble is that when $f\in C^{(k_{1},k_{2},\cdots,k_{n})}(\mathbb{R}\times U,V)$, $\dfrac{df^{\lor}}{dr}$ may not exist.
	To prove $\dfrac{df^{\lor}}{dr}$ exists and $\dfrac{df^{\lor}}{dr}=\left(\dfrac{\partial f}{\partial r}\right)^{\lor}$, we only need to show $\forall r\in \mathbb{R}$, the curve 
	\begin{equation}
		\begin{array}{cccc}
			c: & \mathbb{R} & \longrightarrow & C^{(k_{2},\cdots,k_{n})}(U,V)\\
			& t & \longmapsto & \begin{cases}
				\dfrac{f^{\lor}(r+t)-f^{\lor}(r)}{t} & t\neq0\\
				\left(\dfrac{\partial f}{\partial r}\right)^{\lor}(r) & t=0
			\end{cases}
		\end{array}
	\end{equation}
	is continuous.  
	Now we are going to prove $\partial^{(l_{2},\cdots,l_{n})}\circ c\in C^{0}(\mathbb{R},C^{0}(U,V)),\forall l\leq k\  and\  r\in \mathbb{R}$, where $\partial^{(l_{2},\cdots,l_{n})}$ is the map defined at (\ref{partial operator}).
	Equivalently, we will show that $\left(\partial^{(l_{2},\cdots,l_{n})}\circ c \right) ^\wedge:\mathbb{R}\times U\longrightarrow V$ is continuous.
	
	For $t\neq 0$ and $u\in U$, we have:
	\begin{equation}
		\begin{array}{ccc}
			\left(\partial^{(l_{2},\cdots,l_{n})}\circ c\right)^{\wedge}(t,u) & = & \partial^{(l_{2},\cdots,l_{n})}(c(t))(u)\\
			& = & \partial^{(l_{2},\cdots,l_{n})}\left(\dfrac{f^{\lor}(r+t)-f^{\lor}(r)}{t}\right)(u)\\
			& = & \dfrac{\partial_{u}^{(l_{2},\cdots,l_{n})}f(r+t,u)-\partial_{u}^{(l_{2},\cdots,l_{n})}f(r,u)}{t}\\
			& = & \int_{0}^{1}\partial_{r}\partial_{u}^{(l_{2},\cdots,l_{n})}f(r+st,u)ds
		\end{array}
	\end{equation}
	For $t=0$, we have:
	\begin{equation}
		\begin{array}{ccc}
			\left(\partial^{(l_{2},\cdots,l_{n})}\circ c\right)^{\wedge}(0,u) & = & \partial^{(l_{2},\cdots,l_{n})}\left(c(0)\right)(u)\\
			& = & \partial^{(l_{2},\cdots,l_{n})}\left(\left(\partial_{r}f\right)^{\lor}(r)\right)(u)\\
			& = & \left(\partial_{u}^{(l_{2},\cdots,l_{n})}\left(\partial_{r}f\right)\right)(r,u)\\
			& = & \partial_{r}\partial_{u}^{(l_{2},\cdots,l_{n})}f(r,u)
		\end{array}
	\end{equation}
	So $\left(\partial^{(l_{2},\cdots,l_{n})}\circ c\right)^{\wedge}(t,u)=\int_{0}^{1}\partial_{r}\partial_{u}^{(l_{2},\cdots,l_{n})}f(r+st,u)ds$ for all $(t,u)$. And this function is continuous because $f\in C^{(k_{1},k_{2},\cdots,k_{n})}(\mathbb{R}\times U,V)$, which completes the proof.
\end{proof}
	\begin{definition}
	Let $U=\Pi_{i=1}^{n}U_{i}$ and $U_i$ be an open set in $[0,\infty)^{m_i-b_i}\times \mathbb{R}^{b_i}$, then the 
	$C^{(k_{1},\cdots,k_{n})}$ compact open topology on $C^{(k_{1},\cdots,k_{n})}(U,V)$ is defined to be the initial topology with respect to the family of mappings
	\begin{equation}
		\begin{array}{cccc}
			\partial^{(l_{1},\cdots,l_{n})}: & C^{(k_{1},\cdots,k_{n})}(U,V) & \longrightarrow & C(U,V)\\
			& g & \longmapsto & \partial_{u}^{(l_{1},\cdots,l_{n})}g:=\frac{\partial^{l_{1}}}{\partial u_{1}^{l_{1}}}\cdots\frac{\partial^{l_{n}}}{\partial u_{n}^{l_{n}}}g
		\end{array}
	\end{equation}
	which consists every $\partial^{(l_{1},\cdots,l_{n})}$ such that $|l_i|\leq k_i\in\mathbb{Z}_{\geq 0}$. Here each $l_i$ is a multi- index.
\end{definition}

\begin{remark}\label{remark about U_ij}
	\begin{equation}
		\alpha\in C^{k_{0}}(\mathbb{R},C^{(k_{1},\cdots,k_{n})}(U,V))\Longleftrightarrow\partial^{(l_{1},\cdots,l_{n})}\circ\alpha\in C^{k_{0}}(\mathbb{R},C(U,V)),\ \forall |l_{i}|\leq k_i
	\end{equation}
\end{remark}
\begin{proof}
	The direction $\Rightarrow$ is obvious since $\partial^{(l_{1},\cdots,l_{n})}$ are smooth linear maps.
	
	For direction $\Leftarrow$, the case $k_0=0$ is just the definition of topology on $C^{(k_{1},\cdots,k_{n})}(U,V)$. If $k_0\geq 1$, then it is enough to show that $\partial^{(l_{1},\cdots,l_{n})}\circ\dfrac{d\alpha}{dr}=\dfrac{d}{dr}\left( \partial^{(l_{1},\cdots,l_{n})}\circ\alpha\right) $.
	For 
	\begin{equation}
		\partial^{(l_{1},\cdots,l_{n})}\left(\dfrac{\alpha(r+h)-\alpha(r)}{h}\right)=\dfrac{\left(\partial^{(l_{1},\cdots,l_{n})}\circ\alpha\right)(r+h)-\left(\partial^{(l_{1},\cdots,l_{n})}\circ\alpha\right)(r)}{h}
	\end{equation}
	we know that when $h\rightarrow 0$, the sequence $\partial^{(l_{1},\cdots,l_{n})}\left(\dfrac{\alpha(r+h)-\alpha(r)}{h}\right)$ converges to $\dfrac{d}{dr}\left( \partial^{(l_{1},\cdots,l_{n})}\circ\alpha(r)\right) $. Then by the definition of topology on $C^{(k_{1},\cdots,k_{n})}(U,V)$, we know that $\dfrac{\alpha(r+h)-\alpha(r)}{h}$ converges when $h\rightarrow 0$. Since $\partial^{(l_{1},\cdots,l_{n})}$ is continuous, we finally get 
	\begin{equation}
		\begin{array}{ccc}
			\left( \partial^{(l_{1},\cdots,l_{n})}\circ\dfrac{d\alpha}{dr}\right) (r)=\partial^{(l_{1},\cdots,l_{n})}\left(\underset{h\rightarrow0}{\lim}\dfrac{\alpha(r+h)-\alpha(r)}{h}\right) & = & \underset{h\rightarrow0}{\lim}\partial^{(l_{1},\cdots,l_{n})}\left(\dfrac{\alpha(r+h)-\alpha(r)}{h}\right)\\
			& = & \dfrac{d}{dr}\left( \partial^{(l_{1},\cdots,l_{n})}\circ\alpha(r)\right) 
		\end{array}
	\end{equation}
	So by $\partial^{(l_{1},\cdots,l_{n})}\circ\alpha\in C^{k_{0}}(\mathbb{R},C(U,V))$, we have $\alpha\in C^{k_{0}}(\mathbb{R},C^{(k_{1},\cdots,k_{n})}(U,V))$
\end{proof}
\begin{corollary}
		Let $U=\Pi_{i=1}^{n-1}U_{i}$ and $U_i$ be an open set in $[0,\infty)^{m_i-b_i}\times \mathbb{R}^{b_i}$, then $f\in C^{(k_{1},\cdots,k_{n})}(\mathbb{R}\times U,V)\Longleftrightarrow f^{\lor}\in C^{k_{1}}\left(\mathbb{R},C^{(k_{2},\cdots,k_{n})}(U,V)\right)$, where 
	\begin{equation}
		\begin{array}{cccc}
			f^{\lor}: & \mathbb{R} & \longrightarrow & C^{(k_{2},\cdots,k_{n})}(U,V)\\
			& r & \longmapsto & u\mapsto f(r,u)
		\end{array}
	\end{equation}
\end{corollary}
\begin{proof}
	$\Leftarrow:$ If $f^{\lor}\in C^{k_{1}}\left(\mathbb{R},C^{(k_{2},\cdots,k_{n})}(U,V)\right)$, then 
	$f^{\lor}\in C^{k_{1}}\left(\mathbb{R},C^{(l_{2},\cdots,l_{n})}(U,V)\right)$ for all  $(l_2,\cdots,l_n)$ which satisfies $|l_i|\leq k_i$. Here each $l_i$ are multi-indexes. Then by theorem \ref{general exponential law}, $f\in C^{(k_{1},l_{2},\cdots,l_{n})}(\mathbb{R}\times U,V)$ for all $(l_2,\cdots,l_n)$ which satisfies $|l_i|\leq k_i$, hence $f\in C^{(k_{1},\cdots,k_{n})}(\mathbb{R}\times U,V)$.
	
	$\Rightarrow:$ If $f\in C^{(k_{1},\cdots,k_{n})}(\mathbb{R}\times U,V)$, then $f\in C^{(k_{1},l_{2},\cdots,l_{n})}(\mathbb{R}\times U,V)$ for all $(l_2,\cdots,l_n)$ which satisfies $|l_i|\leq k_i$. So by theorem \ref{general exponential law}, $f^{\lor}\in C^{k_{1}}\left(\mathbb{R},C^{(l_{2},\cdots,l_{n})}(U,V)\right)$ and then 
	$\partial^{(l_{2},\cdots,l_{n})}\circ f^{\lor}\in C^{k_{1}}\left(\mathbb{R},C(U,V)\right)$. So by remark \ref{remark about U_ij}, we have $f^{\lor}\in C^{k_{1}}\left(\mathbb{R},C^{(k_{2},\cdots,k_{n})}(U,V)\right)$.
\end{proof}
\begin{definition}
	Let $M$, $N$ be manifold with corners. If $N=N_1\times \cdots\times N_n$, 
	then we can define $C^{(k_1,\cdots,k_n)}$ maps from $N$ to $M$. A map $f:N\longrightarrow M$ is called $C^{(k_1,\cdots,k_n)}$, if under any local coordinate of type $(U_1\times \cdots U_n,\varphi_1\times\cdots\times\varphi_n)$ on $N$ and any local coordinate $(U_{M},\varphi_{M})$ on $M$, $\varphi_{M}\circ f\circ(\varphi_1\times\cdots\times\varphi_n)^{-1}$ is $C^{(k_1,\cdots,k_n)}$.
\end{definition}
\begin{definition}[topological linear space structure on $\Gamma_\gamma^k(N,E)$]\label{definition of the section space}
	Let $N,M$ be manifold with corners, $\pi:E\longrightarrow M$ be a smooth vector bundle which fibre is $V$, and $\gamma:N\longrightarrow M$ be a $C^k$ map, where $k$ may be multi-index. Then suppose $\{(W_\alpha,\varphi_\alpha)\}_{\alpha\in A}$ is an atlas of $M$,   
	such that there exists $\Psi_\alpha$ which is a local trivialization of $E$. Take $\{U_\beta)\}_{\beta\in B}$ to be a refinement of the open cover $\{\gamma^{-1}(W_\alpha)\}_{\alpha\in A}$ of $N$, such that $\{(U_\beta,\varphi_\beta)\}$ is a coordinate system on $N$ and $\gamma(U_\beta)\subseteq W_{\rho(\beta)}$. 
	The topological linear space structure on $\Gamma_\gamma^k(N,E)$ is defined by the closed linear embedding:
	\begin{equation}
		\begin{array}{cccc}
			i: & \Gamma_{\gamma}^{k}(N,E) & \longrightarrow & \underset{\beta}{\prod}C^{k}(U_{\beta},V)\\
			& s & \longmapsto & (Pr_{2}\circ\Psi_{\rho(\beta)}\circ(s|_{U_{\beta}}))
		\end{array}
	\end{equation}
\end{definition}
\begin{lemma}\label{curves to a bundle}
	$\alpha\in C^{k_{0}}(\mathbb{R},\Gamma_{\gamma}^{k}(N,E))$ is equivalent to $\alpha^{\wedge}\in\Gamma_{\gamma'}^{(k_{0},k)}(\mathbb{R}\times N,E)$, where $\gamma,\ N,\ E$ are defined as in definition \ref{definition of the section space}, $k$ may be multi-index and 
	\begin{equation}
		\begin{array}{cccc}
			\gamma': & \mathbb{R}\times N & \longrightarrow & M\\
			& (r,n) & \longmapsto & \gamma(n)
		\end{array}
	\end{equation}
\end{lemma}
\begin{proof}
	By definition \ref{definition of the section space}, 
	\begin{equation}
		\alpha\in C^{k_{0}}(\mathbb{R},\Gamma_{\gamma}^{k}(N,E))\Longleftrightarrow i_{\beta}\circ\alpha\in C^{k_{0}}(\mathbb{R},C^{k}(U_{\beta},V),\ \forall\beta\in B
	\end{equation}
	where $U_\beta$ is defined as in definition \ref{definition of the section space}, and $i_{\beta}\circ\alpha:r\longmapsto(n\mapsto Pr_{2}\circ\Psi_{\rho(\beta)}\circ\alpha(r)(n))$. Then by theorem \ref{general exponential law},
	\begin{equation}
		i_{\beta}\circ\alpha\in C^{k_{0}}(\mathbb{R},C^{k}(U_{\beta},V)\Longleftrightarrow\left(i_{\beta}\circ\alpha\right)^{\wedge}\in C^{(k_{0},k)}(\mathbb{R}\times U_{\beta},V)
	\end{equation}
	We also have
	\begin{equation}
		\alpha^{\wedge}\in\Gamma_{\gamma'}^{(k_{0},k)}(\mathbb{R}\times N,E)\Longleftrightarrow Pr_{2}\circ\Psi_{\rho(\beta)}\circ(\alpha^{\wedge}|_{\mathbb{R}\times U_{\beta}})\in C^{(k_{0},k)}(\mathbb{R}\times U_{\beta},V),\ \forall\beta\in B
	\end{equation}
	Finally, notice that $Pr_{2}\circ\Psi_{\rho(\beta)}\circ(\alpha^{\wedge}|_{\mathbb{R}\times U_{\beta}})(r,n)=\left(i_{\beta}\circ\alpha\right)^{\wedge}(r,n)=Pr_{2}\circ\Psi_{\rho(\beta)}(\alpha(r)(n))$ which completes the proof.
\end{proof}
\begin{definition}[definition of $C^{(k_1,\cdots,k_n)}$ compact open topology]
	Let $N=N_1\times \cdots\times N_n$ be a manifold with corners, and $M$ a manifold, $k_i\geq 0$. For $f\in C^{(k_1,\cdots,k_n)}(N,M)$, coordinate chart $(U,\varphi)$ on $N$, $(V,\psi)$ on $M$, $K\subseteq U$ compact subset with $f(K)\in V$, $\varepsilon>0$, define the set 
\begin{equation}
\begin{array}{ccc}
	 & U^{k}(f,\varphi,U,\psi,V,K,\varepsilon):=\{g\in C^{(k_{1},\cdots,k_{n})}(N,M)|g(K)\subseteq V,\\
	& \underset{|l_{i}|\leq k_{i}}{\max}\underset{x\in\varphi(K)}{\sup}||\partial^{(l_{1},\cdots,l_{n})}(\psi\circ g\circ\varphi^{-1})(x)-\partial^{(l_{1},\cdots,l_{n})}(\psi\circ f\circ\varphi^{-1})(x)||<\varepsilon\}\\
\end{array}
\end{equation}
	where $||\bullet||$ is the Euclidean norm, the compact open $C^{(k_1,\cdots,k_n)}$ topology is the topology generated by the set
	\begin{equation}
		\begin{array}{ccc}
			&\{U^{k}(f,\varphi,U,\psi,V,K,\varepsilon)|f\in C^{(k_{1},\cdots,k_{n})}(N,M),(U,\varphi)\ and\ (V,\psi)\ charts\ of\ N\ and\ M,\\
			&K\subseteq U\ compact\ with\ f(K)\subseteq V,\varepsilon>0\}\\
		\end{array}
	\end{equation}
	as a subbasis.
\end{definition}
\begin{remark}\label{convergence of g(f_n)}
	Let $f_n\in C^{(k_{1},\cdots,k_{n})}(N,M)$ be a sequence that converges to $f\in C^{(k_{1},\cdots,k_{n})}(N,M)$, $g\in C^\infty (M,M')$, then $g\circ f_n$ converges to $g\circ f$.
\end{remark}
\begin{proof}
	Let $(U,\varphi)$ be a chart of $N$, $(V',\psi')$  be a chart of $M'$. 
	We only need to prove $\partial^{(l_{1},\cdots,l_{n})}(\psi'\circ g\circ f_n\circ\varphi^{-1})$ converges uniformly to $\partial^{(l_{1},\cdots,l_{n})}(\psi'\circ g\circ f\circ\varphi^{-1})$ on a compact neighborhood $K_x$ of any $x\in N$ for any $(l_{1},\cdots,l_{n})$ satisfies $|l_i|<k_i$.
	
	Choose a coordinate chart $(V,\psi)$ near $f(x)$, and choose $K_x\subseteq N$ such that $f(K_x)\in V, g(f(K_x))\in V'$. Because $f_n\rightarrow  f$ uniformly on $K_x$, we have an integer $Z>0$ such that $\forall n>Z$, $f_n(K_x) \subseteq V, g(f_n(K_x))\in V'$, and $\psi(f_n(K_x))\subseteq \overline{\psi(f(K_x))_\varepsilon}\subseteq \mathbb{R}^{dim\ M}$. Here $\psi(f(K_x))_\varepsilon:=\{y\in \mathbb{R}^{dim\ M}|d(y,\psi(f(K_x)))<\varepsilon\}$ is the $\varepsilon$ neighborhood of $\psi(f(K_x))$.
	
	Then by the smoothness of $g$ and the compactness of $\overline{\psi(f(K_{x}))_{\varepsilon}}$, we get
	 \begin{equation}
		\underset{|l_{i}|\leq k_{i}}{\max}\underset{y\in\overline{\psi(f(K_{x}))_{\varepsilon}}}{\sup}\ensuremath{\partial^{(l_{1},\cdots,l_{n})}(\psi'\circ g\circ\psi^{-1})(y)}<C
	\end{equation}
	And then by the chain rule and the convergence of $f_n$, we have
	\begin{equation}
	\underset{x\in K_{x}}{\sup}\partial^{(l_{1},\cdots,l_{n})}\left((\psi'\circ g\circ\psi^{-1})\circ(\psi\circ f_{n}\circ\varphi^{-1}-\psi\circ f\circ\varphi^{-1})\right)(x)\longrightarrow 0
	\end{equation} 
	as $n\longrightarrow \infty$, hence $\partial^{(l_{1},\cdots,l_{n})}(\psi'\circ g\circ f_n\circ\varphi^{-1})$ converges uniformly to $\partial^{(l_{1},\cdots,l_{n})}(\psi'\circ g\circ f\circ\varphi^{-1})$ on $K_x$, for arbitrary $x\in N$.
	
\end{proof}
\begin{theorem}\label{banach manifold consists of C^k maps}
	Let $N$ be an compact manifold with corners, $M$ a manifold, then $C^{(k_1,\cdots,k_n)}(N,M)$ is a $C^\infty$ Banach or $Fr\acute{e}chet$ manifold.
\end{theorem}	
\begin{proof}
	Let $O\subseteq TM$ be an open neighborhood of the zero section, such that the map $(\pi_M\times Exp)|_{O}$ is an embedding, where $Exp$ is the exponential map, and $\pi_M\times Exp:(m,v)\mapsto (m,Exp_m(v))\in M\times M$.
	$\forall \gamma\in C^{\infty}(N,M)$, let
	\begin{equation}
		\begin{array}{cccc}
			exp_{*}: & \Gamma_{\gamma}^{(k_{1},\cdots,k_{n})}(N,O)\subseteq\Gamma_{\gamma}^{(k_{1},\cdots,k_{n})}(N,TM) & \longrightarrow & U_{\gamma}\subseteq C^{(k_{1},\cdots,k_{n})}(N,M)\\
			& x\mapsto\beta(x) & \longmapsto & x\mapsto Exp_{\gamma(x)}\beta(x)
		\end{array}
	\end{equation}
	\begin{equation}
		\begin{array}{cccc}
			exp_{*}^{-1}: & U_{\gamma}\subseteq C^{(k_{1},\cdots,k_{n})}(N,M) & \longrightarrow & \Gamma_{\gamma}^{(k_{1},\cdots,k_{n})}(N,TM)\\
			& x\mapsto f(x) & \longmapsto & x\mapsto Exp_{\gamma(x)}^{-1}f(x)
		\end{array}
	\end{equation}
	Here $U_\gamma$ is the image of $\Gamma_{\gamma}^{(k_{1},\cdots,k_{n})}(N,O)$ under $exp_{*}$. $U_\gamma=\{f\in C^{(k_{1},\cdots,k_{n})}(N,M)|Exp_{\gamma(x)}^{-1}f(x)\in O, \forall x\in N\}$
	Now we are going to prove that $exp_{*}$ is a homeomorphism.
	Let $\{\beta_{i}\}_{i=1}^{\infty}$ be a sequence that converge to $\beta_0$. Then because $Exp$ is $C^\infty$, we have $Exp\circ \beta_{i}$ converges to $Exp\circ\beta_0$. So $exp_*$ is continuous.
	For $exp_{*}^{-1}$, it can be written as the composition of two continuous maps:
	\begin{equation}
		\begin{array}{cccccc}
			exp_{\gamma}^{-1}: & U_{\gamma}\subseteq C^{(k_{1},\cdots,k_{n})}(N,M) & \longrightarrow & C^{(k_{1},\cdots,k_{n})}(N,M\times M) & \overset{\left(\pi_{M}\times Exp\right)^{-1}_{*}}{\longrightarrow} & \Gamma_{\gamma}^{(k_{1},\cdots,k_{n})}(N,TM)\\
			& x\mapsto f(x) & \longmapsto & x\mapsto(\gamma(x),f(x)) &  & x\mapsto Exp_{\gamma(x)}^{-1}f(x)
		\end{array}
	\end{equation}
	Now we are going to prove the smoothness of the transition map.
	\begin{equation}
		\begin{array}{cccc}
			exp_{\gamma_{2}}^{-1}\circ exp_{\gamma_{1}}: & U\subseteq\Gamma_{\gamma_{1}}^{(k_{1},\cdots,k_{n})}(N,TM) & \longrightarrow & \Gamma_{\gamma_{2}}^{(k_{1},\cdots,k_{n})}(N,TM)\\
			& x\mapsto f(x) & \longmapsto & x\mapsto Exp_{\gamma_{2}(x)}^{-1}Exp_{\gamma_{1}(x)}f(x)
		\end{array}
	\end{equation}
	Use the isomorphism of Banach (or $Fr\acute{e}chet$) space
	\begin{equation}
		\begin{array}{cccc}
			iso: & \Gamma_{\gamma}^{(k_{1},\cdots,k_{n})}(N,TM) & \simeq & \Gamma_{\gamma\times id_{N}}^{(k_{1},\cdots,k_{n})}(N,TM\times N)\\
			& f & \longmapsto & f\times id_{N}
		\end{array}
	\end{equation}
	the smoothness of transition map is equal to the smoothness of
	\begin{equation}
		\begin{array}{cccc}
			\Phi'_{\gamma_{1},\gamma_{2}}: & U'\subseteq\Gamma_{\gamma_{1}\times id_{N}}^{(k_{1},\cdots,k_{n})}(N,TM\times N) & \longrightarrow & \Gamma_{\gamma_{2}\times id_{N}}^{(k_{1},\cdots,k_{n})}(N,TM\times N)\\
			& x\mapsto\left(f(x),x\right) & \longmapsto & x\mapsto\left(Exp_{\gamma_{2}(x)}^{-1}Exp_{\gamma_{1}(x)}f(x),x\right)
		\end{array}
	\end{equation}
	So let
	\begin{equation}
		\begin{array}{cccc}
			\varphi_{\gamma_{1},\gamma_{2}}: & U_{12}\subseteq TM\times N & \longrightarrow & TM\times N\\
			& (m,v,x) & \longmapsto & \left(Exp_{\gamma_{2}(x)}^{-1}Exp_{m}v,x\right)
		\end{array}
	\end{equation}
	we have $\Phi'_{\gamma_{1},\gamma_{2}}(f\times id_{N})=\varphi_{\gamma_{1},\gamma_{2}}\circ(f\times id_{N})$.
	
	Notice that $\varphi_{\gamma_{1},\gamma_{2}}$ is $C^\infty$ because it is the composition of the following smooth maps:
	\begin{equation}
		\begin{array}{ccccccc}
			TM\times N & \overset{Exp\times id_{N}}{\longrightarrow} & M\times N & \overset{\gamma_{2}\times id_{M}\times id_{N}}{\longrightarrow} & M\times M\times N & \overset{\left(\pi_{M}\times Exp\right)^{-1}\times id_{N}}{\longrightarrow} & TM\times N\\
			(m,v,x) & \longmapsto & (Exp_{m}v,x)\\
			&  & (m,x) & \longmapsto & (\gamma_{2}(x),m,x) & \longmapsto & \left(Exp_{\gamma_{2}(a)}^{-1}m,x\right)
		\end{array}
	\end{equation}	
	We only need to prove $\Phi'_{\gamma_{1},\gamma_{2}}$ maps $C^\infty$ curves into $C^\infty $ curves. By lemma \ref{curves to a bundle}, 
	a curve $\alpha\in C^{\infty}\left(\mathbb{R},\Gamma_{\gamma_{1}\times id_{N}}^{(k_{1},\cdots,k_{n})}(N,TM\times N)\right)$ if and only if $\alpha^{\wedge}\in C^{(\infty,k_{1},\cdots k_{n})}(\mathbb{R}\times N,TM\times N)$ and $\pi\circ\alpha ^\wedge(r,x)=(\gamma(x),x)$.
	Then because  $(\Phi'_{\gamma_{1},\gamma_{2}}\circ\alpha)^\wedge=\varphi_{\gamma_{1},\gamma_{2}}\circ \alpha^\wedge $, and $\varphi_{\gamma_{1},\gamma_{2}}$ is $C^\infty$, we know that $(\Phi'_{\gamma_{1},\gamma_{2}}\circ\alpha)^\wedge$ is $C^{(\infty,k_{1},\cdots k_{n})}$, which completes the proof.
\end{proof}

\begin{remark}
	Fix a point $p\in N$ and $q\in M$, let $C_*^{k}(N,M):=\{f\in C^{k}(N,M)|f(p)=q\}$, then $C_*^{k}(N,M)$ is a embedded submanifold of $C^{k}(N,M)$, where $k$ may be multi-index.
\end{remark}
\begin{proof}
	For any $f\in C_*^{k}(N,M)$, choose a coordinate chart $(U_f,exp_f^{-1})$ of the total space $C^{k}(N,M)$, where 
	\begin{equation}
		\begin{array}{cccc}
			exp_{f}^{-1}: & U_{f}\subseteq C^{k}(N,M) & \longrightarrow & \Gamma_{f}^{k}(N,TM)\\
			& x\mapsto\beta(x) & \longmapsto & x\mapsto Exp_{f(x)}^{-1}\beta(x)
		\end{array}
	\end{equation}
	And then we have $C_*^{k}(N,M)\cap U_f =\{g\in U_f|exp_f^{-1}(g)(p)=(q,0)\in TM\}$. 
	
	Therefore $exp_f^{-1}\left( C_*^{k}(N,M)\cap U_f\right)=V\cap exp_f^{-1}(U_f)$, where $V:=\{h\in \Gamma_{f}^{k}(N,TM)|h(p)=(q,0)\}$ is a closed subspace of $\Gamma_{f}^{k}(N,TM)$. So $C_*^{k}(N,M)$ is a embedded submanifold which locally homeomorphic to an open set in $V$.
\end{proof}

\begin{theorem}\label{exponential law on manifold}
	Let $M$ be a manifold, and $N$ be a compact manifold with corners, then $\alpha\in C^{\infty}(\mathbb{R},C^{k}(N,M))$ if and only if $\alpha^\wedge\in C^{(\infty,k)}(\mathbb{R}\times N,M)$, where $k$ may be multi-index $(k_1,\cdots,k_n)$ and $\alpha^\wedge:(r,n)\mapsto \alpha(r)(n)$. 
\end{theorem}
\begin{proof}
	Let $\alpha \in C^\infty\left(\mathbb{R},C^k(N,M) \right) $. 
	Let $O\subseteq TM$ be a neighborhood of the zero section, such that $\pi_{M}\times Exp$ is a diffeomorphism on $O$, where
	\begin{equation}
		\begin{array}{cccc}
			\pi_{M}\times Exp: & TM & \longrightarrow & M\times M\\
			& (m,v) & \longmapsto & (m,Exp(m,v))
		\end{array}
	\end{equation}
	For each $s\in \mathbb{R},\ \exists\delta>0$, such that $\alpha((s-\delta,s+\delta ))\subseteq U_\gamma\subseteq C^k(N,M)$, where $(U_\gamma,\varphi_\gamma)$ is a local coordinate chart near $\alpha(s)$. So $\varphi_{\gamma}\circ\alpha\in C^{\infty}((s-\delta,s+\delta),\Gamma_{\gamma}^{k}(N,O))$, which means $\left(\varphi_{\gamma}\circ\alpha\right)^{\wedge}\in\Gamma_{\gamma'}^{(\infty,k)}((s-\delta,s+\delta)\times N,O))$, here $\gamma':(s,n)\mapsto\gamma(n)$. But we have $\left(\varphi_{\gamma}\circ\alpha\right)^{\wedge}:(r,n)\longmapsto Exp_{\gamma(n)}^{-1}\alpha(r)(n)$, so $\alpha^{\wedge}(r,n)=Exp\circ\left(\varphi_{\gamma}\circ\alpha\right)^{\wedge}\in C^{(\infty,k)}$.
	
	On the other hand, if $\alpha^\wedge\in C^{(\infty,k)}(\mathbb{R}\times N,M)$, 
	then for any given $s\in \mathbb{R}$, the map $\alpha(s):n\rightarrow \alpha(s)(n)$ is in $C^{k}(N,M)$, so we can choose a coordinate $(U_\gamma,\varphi_\gamma)$ near $\alpha(s)$. Here $\gamma\in C^\infty(N,M)$ and $U_\gamma$ is defined as in theorem \ref{banach manifold consists of C^k maps}. because $\pi_M\times Exp$ is a diffeomorphism on $O\subseteq TM$, it is an open map, hence $(\pi_M\times Exp)(O)$ is an open subset in $M\times M$. Let $h$ be an $C^{(\infty,k)}$ map:
	\begin{equation}
		\begin{array}{cccc}
			h: & \mathbb{R}\times N & \longrightarrow & M\times M\\
			& (r,n) & \longrightarrow & (\gamma(n),\alpha(r)(n))
		\end{array}
	\end{equation}
	then $h^{-1}\left((\pi_{M}\times Exp)(O)\right)\subseteq\mathbb{R}\times N$ is an open set. And because $\alpha(s)\in U_\gamma$, we know that $Exp^{-1}_{\gamma(n)}\alpha(s)(n)\in O$. And by $(\pi_{M}\times Exp)(Exp^{-1}_{\gamma(n)}\alpha(s)(n))=(\gamma(n),\alpha(s)(n))$, we have $\{s\}\times N\subseteq h^{-1}\left((\pi_{M}\times Exp)(O)\right)$. By the compactness of $N$, we have $\delta>0$, such that $h\left((s-\delta,s+\delta)\times N\right)\subseteq(\pi_{M}\times Exp)(O)$. 
	Since $\pi_{M}\times Exp$ is a diffeomorphism on $O$, we know that $\left( \pi_{M}\times Exp\right) ^{-1}\circ h\in C^{(\infty,k)}((s-\delta,s+\delta)\times N,O)$. And notice that $\left( \varphi_{\gamma}\circ \alpha\right) ^\wedge(r,n)=Exp_{\gamma(n)}^{-1}\alpha(r)(n)=\left( \pi_{M}\times Exp\right) ^{-1}\circ h(r,n)$, so $\left( \varphi_{\gamma}\circ \alpha\right) ^\wedge\in C^{(\infty,k)}$. Hence $\varphi_{\gamma}\circ \alpha\in C^{\infty}((s-\delta,s+\delta),\Gamma_{\gamma}^k(N,O))$, which means $\alpha$ is $C^{\infty}$ near $s$.
\end{proof}

\section{Path space of a manifold}
\begin{theorem}\label{curve P is smooth}
	$P:C_{p}^{k}(I,M)\longrightarrow C^{k-1}(I,T_{p}M)$ is a $C^{\infty}$ map for $k\geq 1$.
\end{theorem}
\begin{proof}
	We only need to show that $P$ maps smooth curves into smooth curves.
	
	By Theorem \ref{exponential law on manifold}, $\alpha\in C^\infty (\mathbb{R},C_{p}^{k}(I,M))$ if and only if $\alpha^\wedge\in C^{(\infty,k)}(\mathbb{R}\times I, M)$. Then by lemma \ref{Parallel transport of smooth vector fields}, $\left( P\circ \alpha\right) ^\wedge(s,t)= P_{\alpha_{s}}^{t\rightarrow0}(\dfrac{\partial\alpha^\wedge}{\partial t}(s,t))$ has $(\infty,k-1)$-th order continuous partial derivatives on any cube $[s_1,s_2]\times I\subseteq \mathbb{R}\times I$. So $\left( P\circ \alpha\right) ^\wedge$ is $C^{(\infty,k-1)}$ on $\mathbb{R}\times I$, and then we know that $P\circ\alpha$ is a smooth curve.
\end{proof}

\begin{theorem}\label{curve P^-1 is smooth}
	$P^{-1}:C^{k-1}(I,T_{p}M)\longrightarrow C_{p}^{k}(I,M)$, which defined by solving equation \ref{global equation for P^-1} is a $C^{\infty}$ map, so P is a diffeomorphism.
\end{theorem}
\begin{proof}
	We are going to prove that the image of a $C^{\infty}$ curve under $P^{-1}$ is a $C^{\infty}$ curve.
	
	Let 
\begin{equation}
	\begin{array}{cccc}
		v: & \mathbb{R}\times I & \longrightarrow & T_{p}M\\
		& s,t & \longmapsto & v(s,t)
	\end{array}
\end{equation}
has $(\infty,k-1)$-th order continuous partial derivatives, so $v$ corresponds to a smooth curve $v^\lor :\mathbb{R}\rightarrow C^{k-1}(I,T_{p}M),\ s\mapsto v(s,-)$. 
Choose a coordinate chart $(U_0,\varphi_0)$ near $p$ and a basis $\{e_i(0)\}\in T_p M$, then we have the equation in $(U_0,\varphi_0)$.
\begin{equation}
	\begin{cases}
		\frac{df_{i}^{l}(s,t)}{dt}=-f_{i}^{j}(s,t)\varGamma_{kj}^{l}(\gamma(s,t))v^{m}(s,t)f_{m}^{k}(s,t)=:g(s,t,\gamma,f)\\
		\dfrac{d\gamma^{k}}{dt}(s,t)=v^{i}(s,t)f_{i}^{k}(s,t)\\
		\gamma^{k}(s,0)=\varphi_{0}(p),\ f_{i}^{j}(s,0)=e_{i}^{j}(0)
	\end{cases}
\end{equation}
By lemma \ref{smoothness of solution to ode}, the solution $(\gamma^{k}(s,t),f_{i}^{j}(s,t))$ has $(\infty,k)$-th order continuous partial derivatives.

By proposition \ref{The equation has the following 3 properties}
\begin{equation}
	\begin{array}{cccc}
		\gamma_{s_{0}}: & I & \longrightarrow & M\\
		& t & \longmapsto & \gamma(s_{0},t)
	\end{array}
\end{equation} 
is a continuous map $\forall s_0 \in \mathbb{R}$. So there exists a partition $0=t_0<t_1<\cdots <t_k=1$ of $[0,1]$, such that $\gamma([t_j,t_{j+1}])\subseteq U_j \subseteq M$, where $U_j$ is a coordinate chart of $M$.

And because the initial value is smooth at $t=0$, 
so $\exists\varepsilon_0 >0$ such that the solution $(\gamma^k,e_i^l)$ exists and is $C^{(\infty,k)}$ on $(s_0-\varepsilon_0,s_0+\varepsilon_0)\times[0,t_1]$.

By induction, we can assume the solution $(\gamma^k,e_i^l)$ exists and is $C^{(\infty,k)}$ on $(s_0-\varepsilon_j,s_0+\varepsilon_j)\times[0,t_{j+1}]$. 
Particularly, $\gamma^k(s,t_{j+1})$ and $e_i^l(s,t_{j+1})$ are $C^\infty$ with respect to $s$.

So we know that $\exists \varepsilon_{j+1}>0$, such that the solution $(\gamma^k,e_i^l)$ exists and is $C^{(\infty,k)}$ on $(s_0-\varepsilon_{j+1},s_0+\varepsilon_{j+1})\times[0,t_{j+2}]$.

Finally, we proved $\forall s_0\in I,\ \exists \varepsilon_{k-1}>0$, such that $(\gamma^k,e_i^l)$ is $C^{(\infty,k)}$ on $(s_0-\varepsilon_{k-1},s_0+\varepsilon_{k-1})\times[0,1]$. So $\gamma(s,t)$ is $C^{(\infty,k)}$ on $\mathbb{R}\times I$, and $\gamma$ corresponds to a $c^\infty $ curve $\gamma^{\vee}:\mathbb{R}\longrightarrow C^{k}_p(I,M)$, which means that $P^{-1}$ maps $C^\infty$ curves into $C^\infty$ curves.
\end{proof}
\begin{remark}
	With the same process as in theorem \ref{curve P is smooth} and theorem \ref{curve P^-1 is smooth}, we have $P:C_{p}^{k}([-1,1],M)\longrightarrow C^{k-1}([-1,1],T_{p}M)$ is a diffeomorphism, where $C_{p}^{k}([-1,1],M):=\{\gamma\in C^{k}([-1,1],M)|\gamma(0)=p\}$, and $P:\gamma\longmapsto t\mapsto P^{t\rightarrow0}_\gamma \dot{ \gamma}(t)$
\end{remark}
\begin{corollary}\label{homeomorphism C(R,M) to C(R,TpM)}
	$C^{k}_p(\mathbb{R},M)$ with its $C^k$ compact open topology is homeomorphic to the $C^{k-1}$ compact open topology of  $C^{k-1}(\mathbb{R},T_p M)$ 
\end{corollary}
Define 
\begin{equation}
	\begin{array}{cccc}
		P': & C_{p}^{k}(\mathbb{R},M) & \longrightarrow & C^{k-1}(\mathbb{R},T_{p}M)\\
		& \gamma & \longmapsto & t\mapsto P_{\gamma}^{t\rightarrow0}\dot{\gamma}(t)
	\end{array}
\end{equation}
For any $i_a:[-1,1]\longrightarrow \mathbb{R}$, $i_a:r\longrightarrow ar$, where $a\neq 0$ is a real number, we have the following commute diagram: 

\begin{tikzcd}
	{C^{k}_p(\mathbb{R},M)} \arrow[r, "P'"] \arrow[d, "i_a^*"] & {C^{k-1}(\mathbb{R},T_p M)} \arrow[d, "\tilde{i_a^*}"] \\
	{C^{k}_p([-1,1],M)} \arrow[r, "a^{-1}P"]                            & {C^{k-1}([-1,1],T_p M)}                          
\end{tikzcd}

where $a^{-1}P:\gamma \longmapsto t\mapsto a^{-1}P_{\gamma}^{t\rightarrow0}\dot{\gamma}(t)$ is just $P$ multiplied by $a^{-1}$.

So $\tilde{i_a^*}\circ P'=a^{-1}P\circ i_a^*$ is continuous, hence $P'$ is continuous by the definition of the $C^{k-1}$ compact open topology.

$P'^{-1}$ is defined by solving equation \ref{global equation for P^-1} and by proposition \ref{The equation has the following 3 properties} (3), we know that it is well defined. 

Then we prove $P'^{-1}$ is continuous, which is equivalent to $\forall a\in \mathbb{R},\ i_a^*\circ P'^{-1}$  is continuous. But $i_a^*\circ P'^{-1}=(a^{-1}P)^{-1}\circ \tilde{i_a^*}$, which is continuous, so we are done.

\begin{definition}[polynomial-like curve]
	A polynomial-like curve $\gamma$ of degree $n$ on $M$ is a $C^k$ curve which $n^{th}$ order covariant derivative is zero, that is, $\nabla_{\dot{ \gamma}}^{(n)}\dot{ \gamma}(t):=\nabla_{\dot{ \gamma}}\cdots\nabla_{\dot{ \gamma}}\dot{ \gamma}(t)=0$ for any $t$. In particular, a geodesic is a polynomial-like curve of degree 1.
\end{definition}
\begin{remark}
	$P$ and $P^{-1}$ maps a polynomial curve into a polynomial curve.
\end{remark}
\begin{proof}
	Only need to show that $P^{t\rightarrow 0}_{\gamma}\left( \nabla_{\dot{ \gamma}}\dot{ \gamma}(t)\right) =\dfrac{d}{d t} P(\gamma)(t)$,  
  which implies that $P^{t\rightarrow 0}_{\gamma}\left( \nabla^{(n)}_{\dot{ \gamma}}\dot{ \gamma}(t)\right) =\dfrac{d^n}{d t^n} P(\gamma)(t)$. 
	
	Choose a basis $\{e_i\}$of $T_{\gamma(0)} M$ and parallel transport along $\gamma$. So we get a parallel frame $\{e_i(t)\}$ along $\gamma$. Then $\dot{ \gamma}(t)=v^{i}(t)e_{i}(t)$ and $P(\gamma)(t)=v^{i}(t)e_{i}(0)\in T_{\gamma(0)}M$. So $\nabla_{\dot{ \gamma}}\dot{ \gamma}(t)=\dfrac{\partial v^{i}}{\partial t}(t)e_{i}(t)$, and then $P\left( \nabla_{\dot{ \gamma}}\dot{ \gamma}\right) (t)=\dfrac{d v^{i}}{d t}(t)e_{i}(0)=\dfrac{d}{d t} P(\gamma)$.
	
	So the map $P$ is commute with covariant derivative, then we know that $P^{-1}$ is commute with covariant derivative. 
\end{proof}
\begin{theorem}[Weierstrass approximation theorem]
	Polynomial-like curves are dense in $C^k(I,M)$ with respect to the $C^k$ compact open topology, $\forall k\geq 0$.
\end{theorem}
\begin{proof}
	If $k\geq 1$, then for any $\gamma\in C^k(I,M)$, $\gamma\in C^k_{\gamma(0)}(I,M)$. Then by the usual Weierstrass approximation theorem, there exists $v_n\in C^{k-1}(I,T_{\gamma(0)} M)$ converges to $P(\gamma)$. 
	Hence $P^{-1}(v_n)$ converges to $\gamma$ because $P$ is a homeomorphism.
	
	If $k=0$, because any continuous curve can be approximate by $C^1$ curves, and Polynomial-like curves are dense in $C^{1}(I,M)$, so we are done. 
\end{proof}
\begin{remark}
	Any curve $\gamma\in C^k(I,M)$ has an Taylor expansion because $\gamma\in C^k_{\gamma(0)}(I,M)$ and $P:C^k_{\gamma(0)}(I,M)\longrightarrow C^{k-1}(I,T_{\gamma(0)} M)$ is a diffeomorphism. So we can define the Taylor series of $\gamma$ by taking Taylor series of $P(\gamma)$ in $C^{k-1}(I,T_{\gamma(0)} M)$.
\end{remark}

\subsection{Generalization in higher dimensions}
Let $[-1,1]^{n}$ be the closed unit ball of dimension n, then we will prove for $k_i\geq 1$, $C^{(k_1,\cdots,k_n)}_{p}([-1,1]^{n},M)\simeq \oplus _{i=1}^{n}C^{(k_1,\cdots,k_{i-1},k_i-1)}([-1,1]^{i},T_{p}M)$, hence $C^{(k_1,\cdots,k_n)}_{p}([-1,1]^{n},M)$ is diffeomorphic to a Frechet (or Banach) space. 

\begin{lemma}\label{existance of P^{-1}}
	Let 
	\begin{equation}
		\begin{array}{cccc}
			v: & [-1,1]^{n+1} & \longrightarrow & TM\\
			& (s_{1},\cdots,s_{n},t) & \longmapsto & v(s,t)
		\end{array}
	\end{equation}
	satisfies:
	
	(1) $\pi_{M}\circ v(s,t)=\gamma_{0}(s)$, that is, $\pi_{M}\circ v(s,t)$ doesn't depend on t, where $\pi_{M}:TM\longrightarrow M$ is the projection.
	
	(2) $v$ is a $C^{(a_1,\cdots,a_n,a_{n+1})}$ vector field, where $a_i\geq 0$.
	
	Then there exists a unique $C^{(a_1,\cdots,a_n,a_{n+1}+1)}$ map
	\begin{equation}
		\begin{array}{cccc}
			\gamma: & [-1,1]^{n+1} & \longrightarrow & M\\
			& (s_{1},\cdots,s_{n},t) & \longmapsto & \gamma(s,t)
		\end{array}
	\end{equation}
	such that $P_{\gamma_{s}}^{t\rightarrow0}\left(\dfrac{\partial\gamma}{\partial t}(s,t)\right)=v(s,t)$
\end{lemma}
\begin{proof}
	Only need to prove $\gamma(s,t)$ is $C^{(a_1,\cdots,a_n,a_{n+1}+1)}$ near any point $(s_0,t_0)\in [-1,1]^{n+1}$. Without loss of generality, we assume $t_0\geq 0$.
	Take a coordinate chart $(U_{0},\varphi_{0})$ near $\gamma_{0}(s_{0})$, and take a frame $\{e_{i}(s)\}$ along $\gamma_{0}$. Then in local coordinate, $\gamma(s,t)$ satisfies the following equation:
	\begin{equation}
	\begin{cases}
		\dfrac{df_{i}^{l}(s,t)}{dt}=-f_{i}^{j}(s,t)\varGamma_{kj}^{l}(\gamma(s,t))v^{m}(s,t)f_{m}^{k}(s,t)=:g(s,t,\gamma,f)\\
		\dfrac{d\gamma^{k}}{dt}(s,t)=v^{i}(s,t)f_{i}^{k}(s,t)\\
		\gamma^{k}(s,0)=\gamma_{0}^{k}(s),f_{i}^{j}(s,0)=e_{i}^{j}(s)
	\end{cases}
	\end{equation}
	$g(s,t,\gamma,f)$ is $C^{(i_1,\cdots,i_n,i_{n+1},\infty,\infty)}$.
	So the solution $\gamma(s,t),\ e_{i}^{j}(s,t)$ is $C^{(a_1,\cdots,a_n,a_{n+1}+1)}$ on the region where the solution exists.
	
	Let $s=s_{0}$, because the map
	\begin{equation}
	\begin{array}{cccc}
		\gamma_{s_{0}}: & [-1,1] & \longrightarrow & M\\
		& t & \longmapsto & \gamma(s_{0},t)
	\end{array}
	\end{equation}
	is continuous, so there exists a partition $0=t_{0}<t_{1}<\cdots<t_{r}=1$ of $I$, such that $\gamma_{s_{0}}\left(\left[ t_{i},t_{i+1}\right]  \right) $ contains in some coordinate chart $(U_{i},\varphi_{i})$ of $M$. For $\gamma_{s_{0}}\left(\left[ t_{0},t_{1}\right]\right)\subseteq U_{0}$, we know that $\gamma^{k}(s,t),e_{i}^{j}(s,t)$ is $C^{(a_1,\cdots,a_n,a_{n+1}+1)}$ on $B_{\varepsilon_{0}}(s_{0})\times[t_{0},t_{1}]$.
	
	By induction, we can assume $\gamma^{k}(s,t),e_{i}^{j}(s,t)$  is $C^{(a_1,\cdots,a_n,a_{n+1}+1)}$ on $B_{\varepsilon_{k}}(s_{0})\times[t_{0},t_{k}]$. 
	And because $\gamma^{k}(s,t),e_{i}^{j}(s,t)$ is $C^{(a_1,\cdots,a_n)}$ with respect to $s$ on $B_{\varepsilon_{k}}(s_{0})\times\{t_{k}\}$, and 
	 $\gamma_{s_{0}}\left(\left[ t_{k},t_{k+1}\right]\right)\subseteq U_{k}$, so $\gamma^{k}(s,t),e_{i}^{j}(s,t)$  is $C^{(a_1,\cdots,a_n,a_{n+1}+1)}$ on $B_{\varepsilon_{k+1}}(s_{0})\times[0,t_{k+1}]$.
	
	In conclusion, we obtain that $\forall s_0,\ \exists\varepsilon>0$, such that $\gamma^{k}(s,t),e_{i}^{j}(s,t)$  is $C^{(a_1,\cdots,a_n,a_{n+1}+1)}$ on $B_{\varepsilon_{r}}(s_{0})\times[0,1]$. By the same method for $t\leq 0$, we know $\gamma$ is $C^{(a_1,\cdots,a_n,a_{n+1}+1)}$ on $[-1,1]^{n+1}$.
\end{proof}
\begin{remark}
	The lemma also holds if $v$ is defined on $\mathbb{R}^{m}\times[-1,1]^{n-m}\times[-1,1]$
	\begin{equation}
		\begin{array}{cccc}
			v: & \mathbb{R}^{m}\times[-1,1]^{n-m}\times[-1,1] & \longrightarrow & TM\\
			& (s_{1},\cdots,s_{n},t) & \longmapsto & v(s,t)
		\end{array}
	\end{equation}
	Because in this case, the map $\gamma(s,t)$ we obtained as before is $C^{(a_1,\cdots,a_n,a_{n+1}+1)}$ near any point.
\end{remark}

\begin{lemma}[Parallel transport of smooth vector fields (2)]\label{Parallel transport of smooth vector fields 2}
	Let
	\begin{equation}
		\begin{array}{cccc}
			\alpha: & [-1,1]^{n+1} & \longrightarrow & M\\
			& (s_{1},\cdots,s_{n},t) & \longmapsto & \alpha(s_{1},\cdots,s_{n},t)
		\end{array}
	\end{equation}
	$X$ is a vector field along $\beta:(r_1,\cdots,r_{l},s_{1},\cdots,s_{n},t)\longmapsto \alpha (s_{1},\cdots,s_{n})$. 
	
	Suppose $\dfrac{\partial \alpha}{\partial t}$ is continuous, $k_i\geq 0$, $X\in C^{(k_1,\cdots,k_{l+n+1})}$ and $\alpha \in C^{(k_{l+1},\cdots,k_{l+n+1})}$
	
	Then $P^{0\rightarrow t}_{\alpha_{s}}X(r,s,t)$ is a $ C^{(k_1,\cdots,k_{l+n+1})}$ vector field.
\end{lemma}
\begin{proof}
	Let $\tilde{\alpha}:(r,s,t,u)\longrightarrow P_{\alpha_{s}}^{0\rightarrow u}X(r,s,t)$.
	Then in local coordinate system $(U,\varphi)$, $\tilde{\alpha}$ satisfies the following equation:
	\begin{equation}
		\begin{cases}
			\dfrac{d}{du}\tilde{\alpha}^{m}(r,s,t,u)=-\dfrac{\partial\alpha^{i}}{\partial t}(s,u)\tilde{\alpha}^{j}(r,s,t,u)\Gamma_{ij}^{m}(\alpha(s,u))=:\tilde{f}(r,s,t,u,\tilde{\alpha})\\
			\tilde{\alpha}(r,s,t,0)=X(r,s,t)
		\end{cases}
	\end{equation}
	
	$\tilde{f}(r,s,t,u,\tilde{\alpha})$ has $(k_1,\cdots,k_{l+n},\infty,k_{l+n+1}-1,\infty)$-th order continuous partial derivatives. By lemma \ref{smoothness of solution to ode}, we know that $\tilde{\alpha}(r,s,t,u)$ is 
	$C^{(k_1,\cdots,k_n,k_{l+n+1},k_{l+n+1})}$ on the region where the solution exists.

	Let $\pi_{M}$ be the projection $TM\longrightarrow M$.
	For $\pi_{M}\circ\tilde{\alpha}(r,s,t,u)=\alpha(s,u)$ is continuous, by lemma \ref{lemma lebesgue number}, $\exists \varepsilon>0$, such that each cube $X_{\varepsilon}\subseteq [-1,1]^{l+n+2}$ with length $\varepsilon$ satisfies $\pi_{M}\circ\tilde{\alpha}(X_{\varepsilon})\subseteq U_i$ for some coordinate neighborhood $U_i$ of $M$.
	
	So if the initial value of $\tilde{\alpha}(r,s,t,u)$ in $X_{\varepsilon}$ is $C^{(k_1,\cdots,k_{l+n+1})}$, then $\tilde{\alpha}(r,s,t,u)$ is $C^{(k_1,\cdots,k_{l+n+1},k_{l+n+1})}$ in $X_{\varepsilon}$. For $\tilde{\alpha}(r,s,t,0)=X(r,s,t)$ is $C^{(k_1,\cdots,k_{l+n+1})}$, we conclude that 
	$\tilde{\alpha}(r,s,t,u)$ is $C^{(k_1,\cdots,k_{l+n+1},k_{l+n+1})}$ in $[-1,1]^{n+2}$. 
	
	So the vector field $P^{0\rightarrow t}_{\alpha_{s}}X(r,s,t)=\tilde{\alpha}(r,s,t,t)$ has $(k_1,\cdots,k_{l+n+1})$-th order continuous partial derivatives. 
\end{proof}
\begin{remark}
	If $\alpha $ is defined on $\mathbb{R}^m\times [-1,1]^{n-m}\times [-1,1]$, 
		\begin{equation}
		\begin{array}{cccc}
			\alpha: & [-1,1]^{n+1} & \longrightarrow & M\\
			& (s_{1},\cdots,s_{n},t) & \longmapsto & \alpha(s_{1},\cdots,s_{n},t)
		\end{array}
	\end{equation}
	the lemma still holds because in this case, $P^{0\rightarrow t}_{\alpha_{s}}X(r,s,t)$ has $(k_1,\cdots,k_{l+n+1})$-th order continuous partial derivatives near any point $(r,s,t)$.
\end{remark}

 Next we define $P:C^{(k_1,\cdots,k_n)}_{p}([-1,1]^{n},M)\simeq \oplus _{i=1}^{n}C^{(k_1,\cdots,k_i-1)}([-1,1]^{i},T_{p}M)$. Let $P_{i}^{s_{i}\rightarrow 0}$ be the parallel transport from $\alpha(s_{1},\cdots,s_{i},\cdots,s_{n})$ to $\alpha(s_{1},\cdots,0,\cdots,s_{n})$, along the curve  $\gamma=\alpha_{s_{1}\cdots\hat{s}_{i}\cdots s_{n}}:=s_{i}\longmapsto\alpha(s_{1},\cdots,s_{i},\cdots,s_{n})$, then $P$ can be define as:
\begin{equation}
	\begin{array}{cccc}
		P: & C_{p}^{(k_1,\cdots,k_n)}([-1,1]^{n},M) & \longrightarrow & \oplus_{i=1}^{n}C^{(k_1,\cdots,k_i-1)}([-1,1]^{i},T_{p}M)\\
		& \alpha(s_{1},\cdots,s_{n}) & \longmapsto & \left(P_{1}^{s_{1}\rightarrow0}\cdots P_{i}^{s_{i}\rightarrow0}\dfrac{\partial\alpha}{\partial s_{i}}(s_{1},\cdots,s_{i},0,\cdots,0)\right)
	\end{array}
\end{equation}
\begin{theorem}\label{higher dimention P is a diffeomorphism}
	$P$ is a diffeomorphism.
\end{theorem}
\begin{proof}
	At first, by lemma \ref{Parallel transport of smooth vector fields}, we know that $P_{1}^{s_{1}\rightarrow0}\cdots P_{i}^{s_{i}\rightarrow0}\dfrac{\partial\alpha}{\partial s_{i}}(s_{1},\cdots,s_{i},0,\cdots,0)$ is a $C^{(k_1,\cdots,k_i-1)}$ vector field along the map $(s_{1},\cdots,s_{i})\longmapsto\alpha(0,\cdots,0)=p$, so the map $P$ is well defined. 
	For smoothness, we need to show that $P$ maps smooth curves to smooth curves.
	A curve $\beta : \mathbb{R}\longrightarrow C_{p}^{(k_1,\cdots,k_n)}([-1,1]^{n},M)$ is smooth if and only if $\beta^\wedge:\mathbb{R}\times [-1,1]^{n}\longrightarrow M,\ (t,s)\longmapsto \beta(t)(s)$ is $C^{(\infty,k_1,\cdots,k_n)}$.
	So we have $P_{1}^{s_{1}\rightarrow0}\cdots P_{i}^{s_{i}\rightarrow0}\dfrac{\partial\beta^\wedge}{\partial s_{i}}(t,s_{1},\cdots,s_{i},0,\cdots,0)$ is a  $C^{(\infty,k_1,\cdots,k_i-1)}$ vector field along the constant map $(t,s_1,\cdots,s_i)\longmapsto p$ by lemma \ref{Parallel transport of smooth vector fields}, which means that $P \circ \beta$ is a smooth curve in $\oplus_{i=1}^{n}C^{(k_1,\cdots,k_i-1)}([-1,1]^{i},T_{p}M)$.
	
	Then we define $P^{-1}$ as follows: 
	
	At first, for any $\left(v_{i} \right)_{i=1}^n \in \oplus_{i=1}^{n}C^{(k_1,\cdots,k_i-1)}([-1,1]^{i},T_{p}M)$, we know that $v_{1}$ corresponds to a $C^{k_1}$ curve $\gamma_{1}$ on $M$ by proposition \ref{The equation has the following 3 properties}.
	
	Next, we parallel transport $v_{2}$ along $\gamma_{1}$ and get a vector field $u_2:(s_1,s_2)\longrightarrow P_{\gamma_1}^{0\rightarrow s_1}v_2 (s_1,s_2)$ along the map $(s_1,s_2)\longrightarrow \gamma_1(s_1)$. Then by lemma \ref{existance of P^{-1}}, there exists a unique $C^{(k_1,k_2)}$ map $\gamma_2:[-1,1]^2\longrightarrow M$ such that $P^{s_2\rightarrow 0}_{\gamma_{s_1}}(\dfrac{\partial\gamma_{2}}{\partial s_{2}}(s_{1},s_{2}))=u_2(s_1,s_2)$
	
	Similarly, assume we do the above operation for $v_i$, and get a $C^{(k_1,\cdots,k_i)}$ map $\gamma_i :[-1,1]^i\longrightarrow M$. Then for $v_{i+1}$, let $u_{i+1}:(s_1,\cdots,s_{i+1})\longmapsto P_{i}^{0\rightarrow s_{i}}\cdots P_{1}^{0\rightarrow s_{1}} v_{i+1}(s_{1},\cdots,s_{i+1})$ and use lemma \ref{existance of P^{-1}}, we get a $C^{(k_1,\cdots,k_{i+1})}$ map $\gamma_{i+1}:[-1,1]^{i+1}\longrightarrow M$.
	
	Finally, we get a $C^{(k_1,\cdots,k_n)}$ map $\gamma_{n}:[-1,1]^{n}\longrightarrow M$ and then we can define $P^{-1}((v_i)_{i=1}^n)=\gamma_{n}$
	
	Next we prove $P^{-1}$ maps smooth curves into smooth curves. Let $\alpha$ be an arbitrary $C^\infty$ curve:
	\begin{equation}
		\begin{array}{cccc}
			\alpha: & \mathbb{R} & \longrightarrow & \oplus_{i=1}^{n}C^{(k_1,\cdots,k_i-1)}([-1,1]^{i},T_{p}M)\\
			& t & \longmapsto & \left(v_{i}(t)\right)
		\end{array}
	\end{equation}
	Then
	\begin{equation}
		\begin{array}{cccc}
			v_{i}^{\wedge}: & \mathbb{R}\times [-1,1]^{i} & \longrightarrow & T_{p}M\\
			& \left(t,s\right) & \longmapsto & v_{i}(t)(s)
		\end{array}
	\end{equation}
	is a $C^{(\infty,k_1,\cdots,k_i-1)}$ map, $\forall i\in\{1,2,\cdots,n\}$.
	
	By lemma \ref{existance of P^{-1}}, $v_{1}^{\wedge}$ corresponds to a unique $C^{k_1}$ map $\gamma_{1}:\mathbb{R}\times [-1,1]^{1}\longrightarrow M$, such that $P^{s_1\rightarrow 0}_{\gamma_{1t}}\dfrac{\partial\gamma_{1}}{\partial s_{1}}(t,s_{1})=v_1(t,s_1)$.
	
	Next, we parallel transport $v_{2}^{\wedge}$ along $\gamma_{1}(t,-)$ and get a vector field $u_2:(t,s_1,s_2)\longrightarrow P_{\gamma_{1t}}^{0\rightarrow s_1}v_2^{\wedge}(t) (s_1,s_2)$. By lemma \ref{Parallel transport of smooth vector fields 2}, $u_2$ is $C^{(\infty,k_1,k_2-1)}$. And use lemma \ref{existance of P^{-1}} again, there exists a unique $C^{(\infty,k_1,k_2)}$ map $\gamma_2:\mathbb{R}\times [-1,1]^2\longrightarrow M$ such that $P^{s_2\rightarrow 0}_{\gamma_{s_1 t}}(\dfrac{\partial\gamma_{2}}{\partial s_{2}}(t,s_{1},s_{2}))=u_2(t,s_1,s_2)$.
	
	Similarly, we know that $\gamma_i:\mathbb{R}\times [-1,1]^i\longrightarrow M$ is $C^{(\infty,k_1,\cdots,k_i)}$, which means that $P^{-1}$ maps smooth curves into smooth curves.
\end{proof}
\begin{corollary}
	$C^\infty_p([-1,1]^n,M)$ is diffeomorphic to $\oplus_{i=1}^{n}C^{\infty}([-1,1]^{i},T_{p}M)$.
\end{corollary}
\begin{corollary}
	$C^{(k_1,\cdots,k_n)}_p(\mathbb{R}^n,M)$ is homeomorphic to $\oplus_{i=1}^{n}C^{(k_1,\cdots,k_i-1)}(\mathbb{R}^{i},T_{p}M)$. 
	
	In particular, $C^{\infty }_p(\mathbb{R}^n,M)$ is homeomorphic to $\oplus_{i=1}^{n}C^{\infty }(\mathbb{R}^{i},T_{p}M)$
\end{corollary}
\begin{proof}
	For any $j_a^i:[-1,1]^i\longrightarrow \mathbb{R}^i$, $j_a^i:v\longrightarrow av$, where $a> 0$ is a real number, we have the following commute diagram: 
	
	\begin{tikzcd}
		{C^{(k_1,\cdots,k_n)}_p(\mathbb{R}^n,M)} \arrow[r, "P'"] \arrow[d, "j_a^{n*}"] & {\oplus_{i=1}^{n}C^{(k_1,\cdots,k_i-1)}(\mathbb{R}^{i},T_{p}M)} \arrow[d, "\tilde{j_a^*}"] \\
		{C^{(k_1,\cdots,k_n)}_p([-1,1]^n,M)} \arrow[r, "a^{-1}P"]                            & {\oplus_{i=1}^{n}C^{(k_1,\cdots,k_i-1)}([-1,1]^{i},T_{p}M)}                          
	\end{tikzcd}
	
	where $a^{-1}P$ is just $P$ multiplied by $a^{-1}$.
	
	So $\tilde{j_a^*}\circ P'=a^{-1}P\circ j_a^{n*}$ is continuous, hence $P'$ is continuous by the definition of the $C^{(k_1,\cdots,k_i-1)}$ compact open topology.
	
	$P'^{-1}$ is defined as in theorem \ref{higher dimention P is a diffeomorphism}.
	
	Then we prove $P'^{-1}$ is continuous, which is equivalent to $\forall a\in \mathbb{R},\ j_a^{n*}\circ P'^{-1}$ is continuous. But $j_a^{n*}\circ P'^{-1}=(a^{-1}P)^{-1}\circ \tilde{j_a^*}$, which is continuous, so we are done.
\end{proof}

\section{Fibre bundles between mapping space}
\begin{theorem}\label{A}
	Let $i:A\hookrightarrow X$ be a cofibration, where $A$ is a compact topological space and $M$ is a manifold. Then $i^{*}:C(X,M)\longrightarrow C(A,M)$ is a locally trivial fiber bundle on any connected component of $C(A,M)$.
\end{theorem}

Choose a complete Riemann metric $g$ on $M$. 
The exponential map is an embedding near the origin, so there is continuous map $r_0:M\longrightarrow \mathbb{R}_{>0}$, such that the restriction of exponential map $Exp|_{B_{r_0(p)}(0)\subseteq T_p M}$ is an embedding.
For $A$ is a compact set, so $\forall f\in C(A,M)$, $\exists \varepsilon=\frac{1}{2}\underset{a\in A}{\inf}\ r_{0}(f(a))>0$, such that $\forall a\in A$, $Exp_{f(a)}:B_{\varepsilon}(0)\subseteq T_{f(a)}M \longrightarrow M$ is an embedding. And then there is a unique shortest geodesic from $f(a)$ to $q$ for all $q$ such that $d(q,f(a))<\varepsilon$.

Since $A$ is compact, and $(M,d_M)$ is a metric space, so $C(A,M)$ is a metric space, $d(f,g):=\underset{a\in A}{\sup}\ d_{M}(f(a),g(a))$, we can take the $\varepsilon$-neighborhood $U_f$ of $f$. And next we are going to prove $U_f \times (i^*)^{-1}(f)\cong(i^*)^{-1}(U_f)$  

For any map $\tilde{f}\in C(X,M)$, $\tilde{f}$ can be viewed as a section from $X$ to the trivial fiber bundle $X\times M$. We  want to construct a automorphism $\varphi_g :X\times M\longrightarrow X\times M$ for every $g\in U_f$, such that $(\varphi_g)_*$ maps the fiber $(i^*)^{-1}(f)$ to the fiber  $(i^*)^{-1}(g)$ homeomorphicly.
$$\begin{array}{cccc}
	(\varphi_{g})_{*}: & \varGamma(X,X\times M) & \longrightarrow & \varGamma(X,X\times M)\\
	& \tilde{f} & \longrightarrow & \varphi_{g}\circ\tilde{f}\end{array}$$

$\varphi_{g}$ can be construct from a family of automorphisms of $M$, denoted by $\varphi_{g,x}:M\longrightarrow M$.
To construct $\varphi_{g,x}$, we construct $\varphi_{p,q}$ first.
\subsection{Construction of $\varphi_{p,q}$}
In this section, we will construct an automrophism $\varphi_{p,q}$ of $M$ for every $p\in M$ and $q$ near $p$, such that $\varphi_{p,q}$ smoothly depend on $(p,q)$ and $\varphi_{p,q}(p)=q$.

Let $q$ be a point near $p$, $d(p,q)<\frac{1}{2}r_0(p)$. $\gamma_{p,q}:I\longrightarrow M$ is the minimal geodesic from $p$ to $q$. $\dfrac{d}{dt}\gamma_{p,q}$ is the tangent vector field along $\gamma_{p,q}$, which can be extend to a compact support vector field $Y(p,q)$ on $M$. Then the corresponding one parameter diffeomorphism group $\varphi_{p,q,t}$ satisfies $\varphi_{p,q,1}(p)=q$. Now we are going to construct $Y(p,q)$, such that
\begin{equation}
	\begin{array}{cccc}
		Y: & U\subseteq M\times M & \longrightarrow & \mathfrak{X}_c(M)\\
		& (p,q) & \longrightarrow & Y(p,q)
	\end{array}
\end{equation} 
is a smooth map, where $U=\{(p,q)\in M\times M\mid d(p,q)<\frac{1}{2}r_0(p)\}$.

Let $q'=Exp_{p}^{-1}(q)\in T_{p}M$ and if $d(m,p)\leq \frac{3}{4}r_0(p)$, let $m'=Exp_{p}^{-1}(m)$.
\begin{proposition}
	\begin{equation}
		\begin{array}{ccccc}
			(\pi_M\times Exp)^{-1}: & U\subseteq M\times M & \longrightarrow & TM & \ is\ a\ C^{\infty}\ map\\
			& (p,q) & \longmapsto & (p,q')
		\end{array}
	\end{equation}
\end{proposition}
\begin{proof}
	The exponential map 
	\begin{equation}
		\begin{array}{cccc}
			\pi_M\times Exp: & TM & \longrightarrow & M\times M\\
			& (p,v) & \longmapsto & (p,Exp_{p}(v))
		\end{array}
	\end{equation}
	is a diffeomorphism when $||v||< r_0(p)$, so the inverse map $(\pi_M\times Exp)^{-1}$ is $C^\infty$ on $U$.
\end{proof}
For every $m\in B_{\frac{3}{4}r_0(p)}(p)$, define $Y_1(p,q)(m)$ as follows:
\begin{equation}\label{definition of Y_1(p,q)(m)}
	Y_1(p,q)(m):=\frac{d}{dt}|_{t=0}Exp_{p}(m'+tq')\in TM
\end{equation}
By the definition of $m'$, the curve $t\mapsto Exp_{p}(m'+tq')$ starts at $m$, so $Y_1(p,q)(m)\in T_mM$.
Let $\rho_2$ be a smooth cut-off function on $M\times M$ such that 
\begin{equation}
	\rho_{2}(p,m)=\begin{cases}
		1 & if\ d(p,m)\leq\dfrac{1}{2}r_{0}(p)\\
		0 & if\ d(p,m)\geq\dfrac{2}{3}r_{0}(p)
	\end{cases}
\end{equation}
Let 
\begin{equation}\label{definition of Y(p,q)}
	Y(p,q)(m)=\begin{cases}
		\rho_{2}(p,m)Y_1(p,q)(m) & if\ d(p,m)\leq\dfrac{2}{3}r_{0}(p)\\
		0 & if\ d(p,m)\geq\dfrac{2}{3}r_{0}(p)
	\end{cases}
\end{equation}
Then $Y(p,q)$ is a compact support vector field and the map $Y:(p,q,m)\mapsto Y(p,q)(m)$ is smooth.

\subsection{Construction of $\varphi_{g,x}$}
\begin{lemma}\label{infinite dimensional exponential map}
 Let $X\in C^{(k_1,\cdots,k_n,\infty)}(U\times M,TM)$, where $U\subseteq[0,+\infty)^{n}$
\begin{equation}
\begin{array}{cccc}
X: & U\times M & \longrightarrow & TM\\
 & (\lambda_{1},\cdots,\lambda_{n},m) & \longmapsto & X(\lambda_{1},\cdots,\lambda_{n},m)
\end{array}
\end{equation}
Suppose for each fixed $\lambda\in U$, $X_{\lambda}:m\mapsto X(\lambda,m)$ is a compact support vector field, and $\varphi(\lambda,t)$ is the one parameter diffeomorphism group corresponds to $X_{\lambda}$, then the map $\varphi^{\wedge}:(\lambda,t,m)\mapsto \varphi(\lambda,t)(m)$ is $C^{(k_1,\cdots,k_n,\infty,\infty)}$. 
\end{lemma}

\begin{proof}
Let
\begin{equation}
\begin{array}{cccc}
\gamma_{\lambda,m}: & \mathbb{R} & \longrightarrow & M\\
 & t & \longmapsto & \varphi(\lambda,t)(m)
\end{array}
\end{equation}
be the integral curve start at $m$. Without loss of generality, we can assume $t'>0$, and we are going to show that $\varphi^{\wedge}$ is $C^{(k_1,\cdots,k_n,\infty,\infty)}$ near $(\lambda',t',m')$. 

Take a partition $0=t_0<t_1<\cdots<t_j=t'+1$ such that each segment $\gamma_{\lambda',m'}([t_i,t_{i+1}])$ is contained in some coordinate chart $(U_i,\psi_i)$ on $M$.
In the local coordinate $(U_0,\psi_0)$ on $M$, we have an ordinary differential equation corresponds to $X$
\begin{equation}
	\begin{cases}
		\dfrac{d\varphi^{\wedge}}{dt}(\lambda,t,m)=X(\lambda,\varphi^{\wedge}(r,t,m))\\
		\varphi^{\wedge}(\lambda,0,m)=m
	\end{cases}
\end{equation}
and $\varphi^{\wedge}(\lambda',t,m')$ exists for $t\in[0,t_1]$. So $\exists \varepsilon_{0}>0$ such that the solution $\varphi^{\wedge}(\lambda,t,m)$ can be defined on $B_{\varepsilon_{0}}(\lambda')\times [0,t_1]\times B_{\varepsilon_{0}}(m')$, where $B_{\varepsilon_{0}}(m')$ is a small ball under local coordinate. By lemma \ref{smoothness of solution to ode}, $\varphi^{\wedge}(\lambda,t,m)$ is $C^{(k_1,\cdots,k_n,\infty,\infty)}$ on $B_{\varepsilon_{0}}(\lambda')\times [0,t_1]\times B_{\varepsilon_{0}}(m')$ because $X$ is $C^{(k_1,\cdots,k_n,\infty)}$.

By induction, we can assume $\varphi^{\wedge}(\lambda,t,m)$ is $C^{(k_1,\cdots,k_n,\infty,\infty)}$ on $B_{\varepsilon_i}(\lambda')\times [0,t_{i+1}]\times B_{\varepsilon_i}(m')$, and we are going to prove that $\exists \varepsilon_{i+1}>0 $, such that $\varphi^{\wedge}(\lambda,t,m)$ is $C^{(k_1,\cdots,k_n,\infty,\infty)}$ on $B_{\varepsilon_{i+1}}(\lambda')\times [0,t_{i+2}]\times B_{\varepsilon_{i+1}}(m')$.

In local coordinate $(U_{i+1},\psi_{i+1})$ on $M$, we have the ordinary differential equation defined for $(\lambda,t,\varphi^{\wedge})\in U\times\mathbb{R}\times U_{i+1}$:
\begin{equation}
		\dfrac{d\varphi^{\wedge}}{dt}(\lambda,t,m)=X(\lambda,\varphi^{\wedge}(\lambda,t,m))\\	
\end{equation}
and the initial value is $\varphi^{\wedge}(\lambda,t_{i+1},m)$, which is $C^{(k_1,\cdots,k_n,\infty)}$ with respect to $(\lambda,m)$ when $(\lambda,m)\in B_{\varepsilon_i}(\lambda')\times B_{\varepsilon_i}(m')$, so $\exists \varepsilon_{i+1}>0$ such that $\varphi^{\wedge}(\lambda,t,m)$ is $C^{(k_1,\cdots,k_n,\infty,\infty)}$ on $B_{\varepsilon_{i+1}}(\lambda')\times [0,t_1]\times B_{\varepsilon_{i+1}}(m')$.

Finally, we know that $\varphi^{\wedge}(\lambda,t,m)$ is $C^{(k_1,\cdots,k_n,\infty,\infty)}$ on $B_{\varepsilon_{j-1}}(\lambda')\times [0,t'+1]\times B_{\varepsilon_{j-1}}(m')$, which completes the proof.
\end{proof}

\begin{lemma}\label{ev is continuous}
	\begin{equation}
		\begin{array}{cccc}
			ev: & C(A,M)\times A & \longrightarrow & M\\
			& (g,a) & \longmapsto & g(a)
		\end{array}
	\end{equation}
	is a continuous map.
\end{lemma}
\begin{proof}
	For $C(A,M)$ and $M$ are metric space, and $\forall (g_0,a_0)\in C(A,M)\times A,\ \forall \varepsilon>0,\  B_{\varepsilon}(g_{0})\times g_{0}^{-1}(B_{\varepsilon}(g_{0}(a_{0})))\subseteq ev^{-1}(B_{2\varepsilon}(g_{0}(a_{0})))$, so the map $ev$ is continuous
\end{proof}
Remember $U_f$ is the $\varepsilon-$neighborhood of $f\in C(A,M)$. Define a map $T:=Y\circ(f\times ev)$:
\begin{equation}
	\begin{array}{cccccc}
		T: & U_{f}\times A & \overset{f\times ev}{\longrightarrow} & M\times M & \overset{Y}{\longrightarrow} & \mathfrak{X}_c(M)\\
		& (g,a) & \longmapsto & \left(f(a),g(a)\right) & \longmapsto & Y\left(f(a),g(a)\right)
	\end{array}
\end{equation}
By lemma \ref{ev is continuous}, $T$ is a continuous map, and $T_{g,a}$ is a compact support smooth vector field. The corresponding one parameter diffeomorphism group $\varphi_{g,a,t}$ satisfies $\varphi_{g,a,1}(f(a))=g(a)$.
Now we get a map:
\begin{equation}
	\begin{array}{cccc}
		\varphi: & U_{f}\times A\times I & \longrightarrow & Diff(M)\\
		& (g,a,t) & \longmapsto & \varphi_{g,a,t}
	\end{array}
\end{equation}
Now we prove that $\varphi$ is continuous because the exponential map $exp:\mathfrak{X}_c(M)\longrightarrow Diff(M)$ is continuous. Here $Diff(M)\subseteq C(M,M)$ has the compact open topology (as a subspace of $C(M,M)$).
We want to prove $exp$ maps convergent sequences into convergent sequences, and because every convergent sequences in $\mathfrak{X}_c(M)$ can be connected with a continuous curve, we only need to prove $exp$ maps continuous curves into continuous curves. Let
$\alpha:\mathbb{R}\longrightarrow \mathfrak{X}_c(M)$
be a continuous curve, then $\alpha^{\wedge}:(r,m)\mapsto \alpha(r)(m)$ is a $C^{(0,\infty)}$ map, so by lemma \ref{infinite dimensional exponential map}, the corresponding diffeomorphism groups $\psi_{r,t}$ satisfies$\psi^{\wedge}:(r,t,m)\mapsto\psi_{r,t}(m)$ is $C^{(0,\infty,\infty)}$. Then $\psi:(r,t)\mapsto\psi_{r,t}$ is continuous. So the map $exp\circ\alpha:r\mapsto \psi_{r,1}$ is continuous.

For $i:A\hookrightarrow X$ is a cofibration, so  
\begin{equation}
	\begin{tikzcd}
		U_{f}\times A\times I\cup U_{f}\times X\times\{0\} \arrow[r, "\varphi\cup id_{M}"] \arrow[d, hook] \arrow[r] & Diff(M) \\
		U_{f}\times X\times I \arrow[ru, "\tilde{\varphi}", dotted]                                                  &        
	\end{tikzcd}
\end{equation}
Here $id_M:U_{f}\times X\times\{0\}\longrightarrow Diff(M)$ is the constant map $(g,x,0)\longmapsto id_M$. Then $\varphi$ can be extend to a continuous map 
\begin{equation}
	\begin{array}{cccc}
		\tilde{\varphi}: & U_{f}\times X\times I & \longrightarrow & Diff(M)\\
		& (g,x,t) & \longmapsto & \tilde{\varphi}(g,x,t)
	\end{array}
\end{equation}
and finally we can let $\varphi_{g,x}:=\tilde{\varphi}(g,x,1)$.
\subsection{The last part of proof}
Now we can prove that $i^{*}:C(X,M)\longrightarrow C(A,M)$ is locally trivial.
Let $f\in C(A,M)$, and $U_f$ be its $\varepsilon$- neighborhood
\begin{equation}
	\begin{array}{cccc}
		F: & U_{f}\times(i^{*})^{-1}(f) & \longrightarrow & (i^{*})^{-1}(U_{f})\\
		& (g,\tilde{f}) & \longmapsto & x\mapsto\varphi_{g,x}(\tilde{f}(x))\\
		F^{-1}: & (i^{*})^{-1}(U_{f}) & \longrightarrow & U_{f}\times(i^{*})^{-1}(f)\\
		& \tilde{g} & \longmapsto & i^{*}(\tilde{g}),\ \ x\mapsto\varphi_{i^{*}(\tilde{g}),x}^{-1}(\tilde{g}(x))
	\end{array}
\end{equation}
Then we prove $F,F^{-1}$ is continuous. We need some lemma first.

\begin{lemma}\label{continuous 1}
	Let $\Phi:U_{f}\times X\longrightarrow C(M,M)$ be a continuous map, then the map
	\begin{equation}
		\begin{array}{cccc}
			\Phi': & U_{f} & \longrightarrow & C(X\times M,M)\\
			& g & \longmapsto & (x,m)\mapsto\Phi(g,x)(m)
		\end{array}
	\end{equation}
	is continuous.
\end{lemma}
\begin{proof}
By lemma \ref{exponential law for Y locally compact}, we have $\Phi^{\wedge}:(g,x,m)\mapsto \Phi(g,x)(m)$ is continuous. So $\Phi'$ is continuous by lemma \ref{exponential law for any topo space}.

	\begin{lemma}\label{continuous 2}
		\begin{equation}
			\begin{array}{cccc}
				\Psi: & C(X,M)\times C(X\times M,M) & \longrightarrow & C(X,M)\\
				& (g,h) & \longmapsto & h\circ(id_{X}\times g)
			\end{array}
		\end{equation}
		is a continuous map.
	\end{lemma}
	\begin{proof}
		Suppose $\Psi (g_0,h_0)\in U^{K}:=\{f\in C(X,M)\mid f(K)\subseteq U\}\subseteq C(X,M)$.
		Let $L:=\Psi (g_0,h_0)(K)\subseteq \overline{L_{\varepsilon}}\subseteq U$, 
		where $L_{\varepsilon}:=\{m\in M|d_M(m,L)<\varepsilon\}$ and $\varepsilon<\frac{1}{2}d_M(L,M\setminus U)$.
		Let $(L_{\varepsilon})^{K}:=\{g\in C(X,M)\mid g(K)\subseteq L_{\varepsilon}\}$ and $U^{K\times\overline{L_{\varepsilon}}}:=\{h\in C(X\times M,M)\mid h(K\times\overline{L_{\varepsilon}})\subseteq U\}$, then for every $(g,h)\in (L_{\varepsilon})^{K}\times U^{K\times\overline{L_{\varepsilon}}}$, $\Psi (g,h)=h\circ(id_{X}\times g)\in U^{K}$, so $\Psi$ is continuous.
	\end{proof}
	Then we can prove 
	\begin{lemma}\label{continuous lemma}
		Let $\Phi:U_{f}\times X\longrightarrow C(M,M)$ be a continuous map, then the map
		\begin{equation}
			\begin{array}{cccc}
				\tilde{F}: & U_{f}\times C(X,M) & \longrightarrow & C(X,M)\\
				& (g,\tilde{f}) & \longmapsto & x\mapsto\Phi(g,x)(\tilde{f}(x))
			\end{array}
		\end{equation}
		is continuous.
	\end{lemma}
	\begin{proof}
		by lemma \ref{continuous 1} and \ref{continuous 2} the map 
		\begin{equation}
			\begin{array}{cccccc}
				\tilde{F}: & U_{f}\times C(X,M) & \overset{\Phi'\times Pr_2}{\longrightarrow} & C(X\times M,M)\times C(X,M) & \overset{\Psi}{\longrightarrow} & C(X,M)\\
				& (g,\tilde{f}) & \longmapsto & ((x,m)\mapsto\Phi(g,x)(m),\tilde{f}) & \longmapsto & x\mapsto\Phi(g,x)(\tilde{f}(x))
			\end{array}
		\end{equation}
		is continuous.
	\end{proof}
	Finally, let 
	\begin{equation}
		\begin{array}{cccc}
			\Phi: & U_{f}\times X & \longrightarrow & C(M,M)\\
			& (g,x) & \longmapsto & \varphi_{g,x}
		\end{array}
	\end{equation}
	and notice that $(i^{*})^{-1}(f)$ and $(i^{*})^{-1}(U_{f})$ are subspace of $C(X,M)$, we conclude that
	\begin{equation}
		\begin{array}{cccc}
			F: & U_{f}\times(i^{*})^{-1}(f) & \longrightarrow & (i^{*})^{-1}(U_{f})\\
			& (g,\tilde{f}) & \longmapsto & x\mapsto\varphi_{g,x}(\tilde{f}(x))\\
		\end{array}
	\end{equation}
	is continuous. And let $\Phi:(g,x)\longmapsto\varphi^{-1}_{g,x}$, we can know that 
	\begin{equation}
		\begin{array}{cccc}
			F^{-1}: & (i^{*})^{-1}(U_{f}) & \longrightarrow & U_{f}\times(i^{*})^{-1}(f)\\
			& \tilde{g} & \longmapsto & i^{*}(\tilde{g}),x\mapsto\varphi_{i^{*}(\tilde{g}),x}^{-1}(\tilde{g}(x))
		\end{array}
	\end{equation}
	is continuous because $i^*$ is continuous and 
	\begin{equation}
		\begin{array}{ccccc}
			(i^{*})^{-1}(U_{f}) & \longrightarrow & U_{f}\times(i^{*})^{-1}(U_{f}) & \longrightarrow & (i^{*})^{-1}(f)\\
			\tilde{g} & \longmapsto & (i^{*}(\tilde{g}),\tilde{g}) & \longmapsto & x\mapsto\varphi_{i^{*}(\tilde{g}),x}^{-1}(\tilde{g}(x))
		\end{array}
	\end{equation}
	is continuous.
\end{proof}

\subsection{Smooth version of Theorem \ref{A}}
In this chapter, we will prove the smooth version of Theorem \ref{A}.
That is, 
\begin{theorem}\label{B}
	Let $M$ be a manifold, $S,N$ be compact manifolds with boundary(or manifolds with corners), and $i:S\hookrightarrow N$ is an inclusion (a $C^k$ injection). Suppose there is an open neighborhood $U_S\subseteq N$ such that $S$ is an $C^k$ retract of $U_S$, which means that there is a $C^k$ map $\pi_{S}:U_S\longrightarrow S$ such that $\pi_{S}\circ i=id_S$. Then $i^{*}:C^{k}(N,M)\longrightarrow C^{k}(S,M)$ is a smooth fiber bundle on any connected component of $C^{k}(S,M)$. Here $k=0,1,2,\cdots,\infty$.	 
\end{theorem}
\begin{remark}
	Let $S,N$ be manifold(without boundary), $i:S\hookrightarrow N$ is an embedding, then $S$ has an tubular neighborhood $U_S$, and $\pi_S:U_S\longrightarrow S$ is just the projection. So the theorem holds when $i:S\hookrightarrow N$ is an embedding between manifolds.
\end{remark}
$\forall f\in C^k(S,M)$, Let $U_f$ be its $\varepsilon$-neighborhood as before, and $\varepsilon<\dfrac{1}{2}\underset{s\in S}{\inf}\{r_{0}(f(s))\}$.

Let $\rho_1:N\longrightarrow \mathbb{R}$ be a $C^\infty$ cut-off function, $\rho_1|_S=1,\ \rho_1|_{N\setminus U_S}=0$. Remember $Y(p,q)$ is a compact support vector field on $M$, such that the corresponding one-parameter diffeomorphism group $\{\varphi_{p,q,t}\}$ satisfies $\varphi_{p,q,1}(p)=q$.

Define 
\begin{equation}
	\begin{array}{cccc}
		X: & U_{f}\times N\times M & \longrightarrow & TM\\
		& (g,n,m) & \longrightarrow & \rho_{1}(n)Y(f(\pi_{S}(n)),g(\pi_{S}(n)))(m)
	\end{array}
\end{equation}
Here $Y(p,q)$ is defined in \ref{definition of Y_1(p,q)(m)} and \ref{definition of Y(p,q)}. Then $X_{g,n}\in \mathfrak{X}_c(M)$, and the one-parameter diffeomorphism group $\varphi_{g,n,t}$ satisfies $\varphi_{g,n,1}(f(s))=g(s),\ \forall s \in S$. So we can define a map:

\begin{equation}
	\begin{array}{cccc}
		F: & U_{f}\times(i^{*})^{-1}(f) & \longrightarrow & (i^{*})^{-1}(U_{f})\\
		& (g,\tilde{f}) & \longmapsto & n\mapsto\varphi_{g,n}(\tilde{f}(n))
	\end{array}
\end{equation}
\begin{equation}
	\begin{array}{cccc}
		F^{-1}: & (i^{*})^{-1}(U_{f}) & \longrightarrow & U_{f}\times(i^{*})^{-1}(f)\\
		& \tilde{g} & \longmapsto & i^{*}(\tilde{g}),n\mapsto\varphi_{i^{*}(\tilde{g}),n}^{-1}(\tilde{g}(n))
	\end{array}
\end{equation}
Now we prove $F,F^{-1}$ are $C^\infty$ maps, that is, $F,F^{-1}$ maps $C^\infty$ curves into $C^\infty$ curves.

For $F$, let $\alpha$ be a smooth curve:
\begin{equation}
	\begin{array}{cccc}
		\alpha: & \mathbb{R} & \longrightarrow & U_{f}\times(i^{*})^{-1}(f)\\
		& u & \longmapsto & \left(\alpha^{1}(u),\alpha^{2}(u)\right)
	\end{array}
\end{equation}
$\alpha^{1}(u)\in C^k(S,M),\ \alpha^{2}(u)\in C^k(N,M)$,then
\begin{equation}
	\begin{array}{cccc}
		(\alpha^{1})^{\wedge}: & \mathbb{R}\times S & \longrightarrow & M\\
		& (u,s) & \longmapsto & \alpha^{1}(u)(s)
	\end{array}\in C^{(\infty,k)}(\mathbb{R}\times S,M)
\end{equation}
\begin{equation}
	\begin{array}{cccc}
		(\alpha^{2})^{\wedge}: & \mathbb{R}\times N & \longrightarrow & M\\
		& (u,s) & \longmapsto & \alpha^{2}(u)(s)
	\end{array}\in C^{(\infty,k)}(\mathbb{R}\times N,M)
\end{equation}
and 
\begin{equation}
	\begin{array}{cccc}
		(F\circ\alpha)^{\wedge}: & \mathbb{R}\times N & \longrightarrow & M\\
		& (u,n) & \longmapsto & \varphi_{\alpha^{1}(u),n}(\alpha^{2}(u)(n))
	\end{array}
\end{equation}
Now we are going to prove $(F\circ\alpha)^{\wedge}\in C^{(\infty,k)}(\mathbb{R}\times N,M)$.

Define $X'(u,n,m):=X_{\alpha^1(u),n}(m)$, we have $X'$ is the composition of the following two maps:
\begin{equation}
	\begin{array}{cccccc}
		X': & \mathbb{R}\times N\times M & \overset{C^{(\infty,k,\infty)}}{\longrightarrow} & M\times M\times M \times N& \overset{c^{\infty}}{\longrightarrow} & TM\\
		& (u,n,m) & \longmapsto & (f(\pi_{S}(n)),(\alpha^{1})^{\wedge}(u,\pi_{S}(n)),m,n)\\
		&  &  & (p,q,m,n) & \longmapsto & \rho_{1}(n)Y(p,q)(m)
	\end{array}
\end{equation}
So we have $X'\in C^{(\infty,k,\infty)}(\mathbb{R}\times N\times M,TM)$. Then by lemma \ref{infinite dimensional exponential map}, $\varphi^{\wedge}(u,n,t,m):=\varphi_{\alpha^{1}(u),n,t}(m)\in C^{(\infty,k,\infty,\infty)}(\mathbb{R}\times N\times \mathbb{R}\times M,M)$.
And because
 $$(F\circ\alpha)^{\wedge}(u,n)=\varphi_{\alpha^{1}(u),n}(\alpha^{2}(u)(n))=\varphi^{\wedge}(u,n,1,\alpha^{2}(u)(n))$$
 so we have $(F\circ\alpha)^{\wedge}\in C^{(\infty,k)}(\mathbb{R}\times N,M)$.

For $F^{-1}$, let 
\begin{equation}
	\begin{array}{cccc}
		\beta: & \mathbb{R} & \longrightarrow & (i^{*})^{-1}(U_{f})\\
		& u & \longmapsto & \beta(u)
	\end{array}
\end{equation}
be a $C^\infty $ curve, then $\beta^\wedge\in C^{(\infty,k)}(\mathbb{R}\times N,M) $
We are going to prove $F^{-1}\circ\beta:u\longmapsto(i^{*}(\beta(u)),(n\mapsto\varphi_{i^{*}(\beta(u)),n}^{-1}(\beta(u)(n))))$ is a $C^\infty$ curve.

Because 
\begin{equation}
	\begin{array}{cccc}
		(i^{*}\circ\beta)^{\wedge}: & \mathbb{R}\times S & \longrightarrow & M\\
		& (u,s) & \longmapsto & \beta(u)(s)
	\end{array}\in C^{(\infty,k)}(\mathbb{R}\times S,M)
\end{equation}
So $i^*\circ\beta$ is a $C^\infty$ curve from $\mathbb{R}$ to $C^k(S,M)$.
Next we prove $(u,n)\longmapsto\varphi_{i^{*}(\beta(u)),n}^{-1}(\beta(u)(n))$ is a $C^{(\infty,k)}$ map.

Let $X''(u,n,m):=X_{i^*(\beta(u)),n}(m)$ then $X''$ is $C^{(\infty,k,\infty)}$ because:
\begin{equation}
	\begin{array}{cccccc}
		X'': & \mathbb{R}\times N\times M & \overset{C^{(\infty,k,\infty)}}{\longrightarrow} & M\times M\times M \times N& \overset{c^{\infty}}{\longrightarrow} & TM\\
		& (u,n,m) & \longmapsto & (f(\pi_{S}(n)),(i^*\circ\beta)^{\wedge}(u,\pi_{S}(n)),m,n)\\
		&  &  & (p,q,m,n) & \longmapsto & \rho_{1}(n)Y(p,q)(m)
	\end{array}
\end{equation}
For  $-X''\in C^{(\infty,k,\infty)}(\mathbb{R}\times N\times M,TM)$, let $\psi(u,n,t):=\varphi_{i^{*}(\beta(u)),n,t}^{-1}$ be the corresponding diffeomorphism groups, then we have $\psi^{\wedge}(u,n,t,m)\in C^{(\infty,k,\infty,\infty)}(\mathbb{R}\times N\times \mathbb{R}\times M,M)$by lemma \ref{infinite dimensional exponential map}.

So finally $(u,n)\longmapsto\psi^{\wedge}(u,n,1,\beta(u)(n))=\varphi_{i^{*}(\beta(u)),n}^{-1}(\beta(u)(n))$ is a $C^{(\infty,k)}$ map, and then $F^{-1}\circ\beta$ is a $C^\infty$ curve. So $F$ is a diffeomorphism, and $i^{*}:C^{k}(N,M)\longrightarrow C^{k}(S,M)$ is a smooth fiber bundle.

\begin{corollary}
	\begin{equation}
		\begin{array}{cccc}
			\pi: & C^{k}(I,M) & \longrightarrow & M\times M\\
			& \gamma & \longmapsto & (\gamma(0),\gamma(1))
		\end{array}
	\end{equation}
	is a smooth fibre bundle.
\end{corollary}

\begin{corollary}\label{vector bundle R to M}
	\begin{equation}
		\begin{array}{cccc}
			\pi: & C^{k}(\mathbb{R},M) & \longrightarrow & M\\
			& \gamma & \longmapsto & \gamma(0)
		\end{array}
	\end{equation}
	is a vector bundle.
\end{corollary}

\begin{proof}
Remember $\varphi_{p,q}$ is a diffeomorphism of $M$ that sends $p$ to $q$, and the map $(p,q,m)\longmapsto \varphi_{p,q}(m)$ is smooth by lemma \ref{infinite dimensional exponential map}. so we can define a homeomorphism for a open neighborhood $U_p$ of arbitrary $p\in M$:
\begin{equation}
\begin{array}{cccc}
F: & U_{p}\times C_{p}^{k}(\mathbb{R},M) & \longrightarrow & \pi^{-1}(U_{p})\\
 & (m,\gamma) & \longmapsto & \varphi_{p,m}\circ\gamma\\
F^{-1}: & \pi^{-1}(U_{p}) & \longrightarrow & U_{p}\times C_{p}^{k}(\mathbb{R},M)\\
 & \sigma & \longmapsto & (\sigma(0),\varphi_{p,\sigma(0)}^{-1}\circ\sigma)
\end{array}
\end{equation}
\end{proof}

\begin{corollary}
	$Imm_p^{k}(\mathbb{R},M)/Diff_+(\mathbb{R})$ is a vector bundle on $S^{n-1}$, where $Imm_p^{k}(\mathbb{R},M)$ consists of immersed curve $\gamma$ with $\gamma(0)=p$, and
	$Diff_+(\mathbb{R})$ consists of orientation preserving automorphism of $\mathbb{R}$ such that $f(0)=0$.
\end{corollary}
\begin{proof}
	Choose a complete Riemann metric $g$ on $M$.
	
	Let $S:=\{\gamma\in C_{p}^{k}(\mathbb{R},M)|\ ||\dot{\gamma}(t)||=1\}$
	and define 
	\begin{equation}
		\begin{array}{cccc}
			r: & Imm_{p}^{k}(\mathbb{R},M) & \longrightarrow & S\\
			& \gamma & \longmapsto & t\mapsto\gamma\left(\int_{0}^{t}\dfrac{1}{||\dot{\gamma}(s)||}ds\right)
		\end{array}
	\end{equation}
	Let $\{\gamma_n\}$ be a sequence converges to $\gamma$ in the $C^k$ compact open topology, which means that the derivatives of $\gamma_n$ converges uniformly on any compact set. We claim that for any $C^k$ maps $f_n,\ g_n$, if $f_n\stackrel{C^{k}}{\longrightarrow}f$ and $g_n\stackrel{C^{k}}{\longrightarrow}g$, then $f_n\circ g_n\stackrel{C^{k}}{\longrightarrow} f\circ g$.
	 By remark \ref*{convergence of g(f_n)}, we have $\dfrac{1}{||\dot{\gamma_n}(t)||} \stackrel{C^{k-1}}{\longrightarrow}\dfrac{1}{||\dot{\gamma}(t)||}$, and then $\int_{0}^{t}\dfrac{1}{||\dot{\gamma_n}(s)||}ds\stackrel{C^{k}}{\longrightarrow}\int_{0}^{t}\dfrac{1}{||\dot{\gamma}(s)||}ds$. So $r(\gamma_{n})\stackrel{C^{k}}{\longrightarrow} r(\gamma)$, and then $r$ is a continuous map. 
	 
	 Now we prove the claim by induction. By taking local coordinates, we can suppose all the maps are defined on open sets in Euclidean space. The case $k=0$ is trivial, suppose we have done for $k=m$, and let $k=m+1$. Then
	\begin{equation}
		\partial_{x_{i}}(f_{n}\circ g_{n})=\underset{j}{\sum}\dfrac{\partial g_{n}^{j}}{\partial x_{i}}(x)\dfrac{\partial f_{n}}{\partial y^{j}}(g_{n}(x))
	\end{equation}
	And by induction, we have $\dfrac{\partial f_{n}}{\partial y^{j}}\circ g_{n}\stackrel{C^{m}}{\longrightarrow}\dfrac{\partial f}{\partial y^{j}}\circ g$ and $\dfrac{\partial g_{n}^{j}}{\partial x_{i}}\stackrel{C^{m}}{\longrightarrow}\dfrac{\partial g^{j}}{\partial x_{i}}$.
	Hence $\partial_{x_{i}}(f_{n}\circ g_{n})\stackrel{C^{m}}{\longrightarrow}\partial_{x_{i}}(f\circ g)$, which completes the proof.
	
	For any equivalent class $[\gamma]\in Imm_p^{k}(\mathbb{R},M)/Diff_+(\mathbb{R})$, there is an unique element in it such that $||\dot{\gamma}(t)||=1,\ \forall t\in\mathbb{R}$. so we can define 
	\begin{equation}
		\begin{array}{cccc}
			\tilde{r}: & C_{p}^{k}(\mathbb{R},M)/Diff_{+}(\mathbb{R}) & \longrightarrow & S\\
			& [\gamma] & \longmapsto & \gamma
		\end{array}
	\end{equation}
	$\tilde{r}$ is continuous because $\tilde{r}\circ p=r$ is continuous and $p:Imm_p^{k}(\mathbb{R},M)\longrightarrow Imm_p^{k}(\mathbb{R},M)/Diff_+(\mathbb{R})$ is a quotient map. $\tilde{r}^{-1}$ is continuous because $\tilde{r}^{-1}=p\circ i$ where $i:S\hookrightarrow Imm_p^{k}(\mathbb{R},M)$.
	So $Imm_p^{k}(\mathbb{R},M)/Diff_+(\mathbb{R})$ is homeomorphic to $S$.
	Then by corollary \ref{homeomorphism C(R,M) to C(R,TpM)}, there is a  homeomorphism $P:C_{p}^{k}(\mathbb{R},M)\longrightarrow C^{k-1}(\mathbb{R},T_p M)$, so we know that $S$ is homeomorphic to $C^{k-1}(\mathbb{R},S^{dim M-1})$, which is a vector bundle on $S^{dim M-1}$ by corollary \ref{vector bundle R to M}.
\end{proof}

\section{Mapping space between compact topological spaces and manifolds} 
Now we prove the mapping space which consists of maps from a compact topological space to a manifold is a $C^\infty $ Banach manifold.
\begin{theorem}\label{C(A,M) is a banach manifold}
	Let $A$ be a compact topological space, $M$ be a $C^\infty$ manifold, then the mapping space $C(A,M)$ with its compact-open topology has a Banach manifold structure.
\end{theorem}
By lemma \ref{exponential law for any topo space} and lemma \ref{exponential law for Y locally compact}, we have:
\begin{lemma}\label{continuous exponential law}
	Let X,Y,Z be topological space and suppose that Y is locally compact. Then $f\in C(X,C(Y,Z))\Longleftrightarrow f^\wedge\in C(X\times Y,Z)$, where $f^\wedge :(x,y)\mapsto f(x)(y)$.
\end{lemma}
\begin{lemma}\label{exponential law}
	Let A be a locally compact topological space, let $X$ be a topological space, let $\pi:E\longrightarrow X$ be a vector bundle (of finite rank), and let $\gamma:A\longrightarrow X$ be a continuous map. Let $\Gamma_\gamma(A,E)$ be the topological linear space of all vector fields along $\gamma$ (with its compact open topology). Then the following two statements are equivalent:
	
	(1)$\alpha\in C^{\infty}(\mathbb{R},\Gamma_\gamma(A,E))$
	
	(2)Let $\alpha^\wedge(r,a):=\alpha(r)(a)$, then $\alpha^\wedge\in C^{(\infty,0)}(\mathbb{R}\times A,E)\ and\ \pi \circ \alpha^\wedge(r,a)=\gamma(a)$
\end{lemma}
\begin{proof}
	We only need to prove $\dfrac{\partial\alpha^{\wedge}}{\partial r}(r,a)=\left(\dfrac{d\alpha}{dr}\right)^{\wedge}(r,a)$.
	If $\dfrac{d\alpha}{dr}$ exists, then 
	\begin{equation}
		\begin{array}{ccc}
			\dfrac{\partial\alpha^{\wedge}}{\partial r}(r,a) & = & \underset{h\rightarrow0}{\lim}\dfrac{\alpha(r+h)(a)-\alpha(r)(a)}{h}\\
			& = & \underset{h\rightarrow0}{\lim}\dfrac{\alpha(r+h)-\alpha(r)}{h}(a)\\
			& = & \dfrac{d\alpha}{dr}(r)(a)
		\end{array}
	\end{equation}
	conversely, if $\dfrac{\partial\alpha^{\wedge}}{\partial r}$ is exist and continuous, then we need to prove $\alpha:\mathbb{R}\longrightarrow\Gamma_\gamma(A,E)$ is a differentiable curve, and $\dfrac{\partial\alpha^{\wedge}}{\partial r}(r,a)=\left(\dfrac{d\alpha}{dr}\right)^{\wedge}(r,a)$.
	Or equivalently,
	\begin{equation}
		\beta:h\longmapsto\begin{cases}
			\dfrac{\alpha(r+h)-\alpha(r)}{h} & h\neq0\\
			a\mapsto\dfrac{\partial\alpha^{\wedge}}{\partial r}(r,a) & h=0
		\end{cases}
	\end{equation}
	is a continuous curve from $(-\varepsilon,\varepsilon)$ to $\Gamma_\gamma(A,E)$ for any $r\in\mathbb{R}$.
	By lemma \ref{continuous exponential law}, we need to prove
	\begin{equation}
		\beta^{\wedge}:\left(h,a\right)\longmapsto\begin{cases}
			\dfrac{\alpha^{\wedge}(r+h,a)-\alpha^{\wedge}(r,a)}{h} & h\neq0\\
			\dfrac{\partial\alpha^{\wedge}}{\partial r}(r,a) & h=0
		\end{cases}
	\end{equation}
	is continuous. But $\beta^{\wedge}(h,a)=\int_{0}^{1}\frac{\partial\alpha^{\wedge}}{\partial r}(r+sh,a)ds$, and $\frac{\partial\alpha^{\wedge}}{\partial r}(r,a)$ is continuous, hence $\beta$ is continuous.
\end{proof}	
\begin{definition}[partially smooth space]
	A m-dimensional partially smooth space is a topological space $M$ such that $\forall x\in M$, there exists $U_x\ni x$, and a homeomorphism $\varphi_x:U_x\longrightarrow V_x\times A_x$. Here $A_x$ is a topological space and $V_x$ is an open set in $\mathbb{R}^m$. Furthermore, for any two coordinate charts $(U_{x_1},\varphi_{x_1}),(U_{x_2},\varphi_{x_2})$, if $U_{x_1}\cap U_{x_2}\neq\phi$, then $\varphi_{x_{1}}\circ\varphi_{x_{2}}^{-1}:V_{x_{2}}\times A_{x_{2}}\longrightarrow V_{x_{1}}\times A_{x_{1}}$ satisfies: $\varphi_{x_{1}}\circ\varphi_{x_{2}}^{-1}(v,a)=(f^{1}(v,a),f^{2}(a))$. Here $f^2$ is a continuous map, and $f^1$ is a $C^{(\infty,0)}$ map.
\end{definition}
\begin{lemma}[composition of partially smooth maps]
	The composition of two  $C^{(\infty,0)}$ maps is a $C^{(\infty,0)}$ map.
\end{lemma}
\begin{proof}
	We prove by induction. Let 
	\begin{equation}
		\begin{array}{cccc}
			f: & U\times A & \longrightarrow & V\times B\\
			& (u,a) & \longmapsto & \left(f^{1}(u,a),f^{2}(a)\right)
		\end{array}
	\end{equation}
	and 
	\begin{equation}
		\begin{array}{cccc}
			g: & V\times B & \longrightarrow & W\times C\\
			& (v,b) & \longmapsto & \left(g^{1}(v,b),g^{2}(b)\right)
		\end{array}
	\end{equation}
	then $g\circ f$ is obviously $C^0$. We want to prove $g\circ f$ is $C^{(\infty,0)}$. 
	Suppose we have proved that $g\circ f$ is $C^{(k,0)}$. Then fix $a$ and use the chain rule, we have:
	\begin{equation}
		\dfrac{\partial (g\circ f)}{\partial u^{i}}=\dfrac{\partial g^{1}}{\partial v}\left(f^{1}(u,a),f^{2}(a)\right)\dfrac{\partial f^{1}}{\partial u^{i}}(u,a)
	\end{equation} 
	So by induction we know that the right hand side is $C^{(k,0)}$, which means that $g\circ f$ is $C^{(k+1,0)}$ and we are done. 
\end{proof}
Now we are going to prove theorem \ref{C(A,M) is a banach manifold}.
\begin{proof}
	Choose a complete Riemann metric $g$ on $M$.
	Let $O\subseteq TM$ be an open neighborhood of the zero section, such that the map $(\pi_M\times Exp)|_{O}$ is an embedding, where $Exp$ is the exponential map, and $\pi_M\times Exp:(m,v)\mapsto (m,Exp_m(v))\in M\times M$.
	Let $\gamma \in C(A,M)$ and $\Gamma_\gamma(A,O):=\{f:A\longrightarrow O|\pi_M\circ f=\gamma\}$ be a subset of $\Gamma_\gamma(A,TM)$ which consists of vector fields along $\gamma$.
	Then we can define 
	\begin{equation}
		\begin{array}{cccc}
			exp_{\gamma}: & \Gamma_{\gamma}(A,O)\subseteq\Gamma_{\gamma}(A,TM) & \longrightarrow & U_{\gamma}\subseteq C(A,M)\\
			& a\mapsto\beta(a) & \longmapsto & a\mapsto Exp_{\gamma(a)}\beta(a)
		\end{array}
	\end{equation}
	and its inverse 
	\begin{equation}
		\begin{array}{cccc}
			exp_{\gamma}^{-1}: & U_{\gamma}\subseteq C(A,M) & \longrightarrow & \Gamma_{\gamma}(A,O)\subseteq\Gamma_{\gamma}(A,TM)\\
			& a\mapsto f(a) & \longmapsto & a\mapsto Exp_{\gamma(a)}^{-1}f(a)
		\end{array}
	\end{equation}
	We need to prove $exp_{\gamma}$ is a homeomorphism. Because $(M,g)$ is a metric space and $A$ is compact, $C(A,M)$ has a metric $d_{C(A,M)}(f_1,f_2):=\underset{a\in A}{\sup}\{d_{M}\left(f_{1}(a),f_{2}(a)\right)\}$. Also, $\Gamma_{\gamma}(A,O)$ is an open set in a Banach space, so it has a metric. For $\beta\in \Gamma_{\gamma}(A,O)$, $Im\beta$ is compact in $TM$, so its $\varepsilon$-neighborhood $U_{Im\beta}$ is a bounded set, and $\overline{U_{Im\beta}}$ is compact. For $Exp$ is locally Lipschitz, $Exp|_{\overline{U_{Im\beta}}}$ is Lipschitz, with Lipschitz constant $L$. If $d_{\Gamma_{\gamma}(A,O)}(\beta',\beta)<\varepsilon$, then $\forall a\in A$, $(\gamma(a),\beta'(a))\in U_{Im\beta}$. Hence $d_{M}(Exp_{\gamma(a)}\beta(a),Exp_{\gamma(a)}\beta'(a))<L\varepsilon$, which means that $exp_{\gamma}$ is locally Lipschitz.

For $f\in U_\gamma\subseteq C(A,M)$, $(\gamma \times f )(A)\subseteq M\times M$ is compact. Let $U_{Im(\gamma\times f)}$ be its $\varepsilon$-neighborhood, then $\overline{U_{Im(\gamma\times f)}}$ is compact.
So $(\pi_M\times Exp)^{-1}$ is Lipschitz on $\overline{U_{Im(\gamma\times f)}}$, with Lipschitz constant $L'$. If $d_{C(A,M)}(f',f)<\varepsilon$, then $\forall a\in A$, $(\gamma(a),f'(a))\in U_{Im(\gamma\times f)}$. Hence $d_{TM}((\pi_M\times Exp)^{-1}(\gamma(a),f(a)),(\pi_M\times Exp)^{-1}(\gamma(a),f'(a)))<L'\varepsilon$, which means that $exp_{\gamma}^{-1}$ is locally Lipschitz. 
	
	Next we are going to prove that the transition function is $C^\infty$. The transition function is:
	\begin{equation}
		\begin{array}{cccc}
			\Phi_{\gamma_{1},\gamma_{2}}:=exp_{\gamma_{2}}^{-1}\circ exp_{\gamma_{1}}: & U\subseteq\Gamma_{\gamma}(A,TM) & \longrightarrow & \Gamma_{\gamma}(A,TM)\\
			& a\mapsto\beta(a) & \longmapsto & a\mapsto Exp_{\gamma_{2}(a)}^{-1}Exp_{\gamma_{1}(a)}\beta(a)
		\end{array}
	\end{equation}
	Because we have the following isomorphism of Banach spaces:
	\begin{equation}
		\begin{array}{cccc}
			iso: & \Gamma_{\gamma}(A,TM) & \simeq & \Gamma_{\gamma\times id_{A}}(A,TM\times A)\\
			& \beta & \longmapsto & \beta\times id_{A}
		\end{array}
	\end{equation}
	the transition function can be viewed as a map
	\begin{equation}
		\begin{array}{cccc}
			\Phi'_{\gamma_{1},\gamma_{2}}: & U'\subseteq\Gamma_{\gamma_{1}\times id_{A}}(A,TM\times A) & \longrightarrow & \Gamma_{\gamma_{2}\times id_{A}}(A,TM\times A)\\
			& a\mapsto\left(\beta(a),a\right) & \longmapsto & a\mapsto\left(Exp_{\gamma_{2}(a)}^{-1}Exp_{\gamma_{1}(a)}\beta(a),a\right)
		\end{array}
	\end{equation}
	Let 
	\begin{equation}
		\begin{array}{cccc}
			\varphi_{\gamma_{1},\gamma_{2}}: & U_{12}\subseteq TM\times A & \longrightarrow & TM\times A\\
			& (v,a) & \longmapsto & \left(Exp_{\gamma_{2}(a)}^{-1}Exp_{\gamma_{1}(a)}v,a\right)
		\end{array}
	\end{equation}
	then $\Phi'_{\gamma_{1},\gamma_{2}}(\beta)=\varphi_{\gamma_{1},\gamma_{2}}\circ \beta$, hence $\Phi'_{\gamma_{1},\gamma_{2}}=\left( \varphi_{\gamma_{1},\gamma_{2}}\right) _*$.
	And we have $\varphi_{\gamma_{1},\gamma_{2}}$ is the composition of the following three $C^{(\infty,0)}$ maps:
	\begin{equation}
		\begin{array}{ccccccc}
			TM\times A & \overset{Exp\times id_{A}}{\longrightarrow} & M\times A & \overset{\gamma_{2}\times id_{M}\times id_{A}}{\longrightarrow} & M\times M\times A & \overset{\left(\pi_{M}\times Exp\right)^{-1}\times id_{A}}{\longrightarrow} & TM\times A\\
			& \longmapsto & (m,a) & \longmapsto & (\gamma_{2}(a),m,a) & \longmapsto & \left(Exp_{\gamma_{2}(a)}^{-1}m,a\right)
		\end{array}
	\end{equation}
	So $\varphi_{\gamma_{1},\gamma_{2}}$ is a $C^{(\infty,0)}$ map.
	To prove the transition function is $C^\infty$, we need to show that it maps $C^\infty $ curves into $C^\infty$ curves.
	Let $\alpha\in C^{\infty}(\mathbb{R},\Gamma_{\gamma_1\times id_{A}}(A,TM\times A))$, then by lemma \ref{exponential law}, we have $\alpha^{\wedge}\in C^{(\infty,0)}(\mathbb{R}\times A,TM\times A)$.
	So $\varphi_{\gamma_{1},\gamma_{2}}\circ\alpha^\wedge$ is a $C^{(\infty,0)}$ map, and by $\left( \Phi'_{\gamma_{1},\gamma_{2}}\circ\alpha \right)^\wedge(r,a)=\varphi_{\gamma_{1},\gamma_{2}}\circ\alpha^\wedge(r,a)=\left(a,Exp_{\gamma_{2}(a)}^{-1}Exp_{\gamma_{1}(a)}\alpha(r)(a)\right)$, we know that $\left( \Phi'_{\gamma_{1},\gamma_{2}}\circ\alpha \right)^\wedge$ is $C^{(\infty,0)}$. Finally we get $\Phi'_{\gamma_{1},\gamma_{2}}(\alpha)\in C^{\infty}(\mathbb{R},\Gamma_{id_{A}\times\gamma_2}(A,TM\times A))$.
\end{proof}

\begin{theorem}
	Let $M$ be a complex manifold, and suppose there is a holomorphic connection on $M$, then for any compact topological space $A$, $C(A,M)$ is a complex Banach manifold.
\end{theorem}
\begin{proof}
	Only need to prove the transition function is holomorphic. Since we have a holomorphic connection, the exponential map is holomorphic.
	The transition function is:
	\begin{equation}
		\begin{array}{cccc}
			\Phi:=exp_{\gamma_{2}}^{-1}\circ exp_{\gamma_{1}}: & U\subseteq\Gamma_{\gamma}(A,T^{(1,0)}M) & \longrightarrow & \Gamma_{\gamma}(A,T^{(1,0)}M)\\
			& a\mapsto\beta(a) & \longmapsto & a\mapsto Exp_{\gamma_{2}(a)}^{-1}Exp_{\gamma_{1}(a)}\beta(a)
		\end{array}
	\end{equation}
	Let $\alpha:D\subseteq\mathbb{C}\longrightarrow \Gamma_{\gamma}(A,T^{(1,0)}M)$ be a holomorphic map, then by the Cauchy-Riemann equation, we have:
	\begin{equation}
		\dfrac{d}{d\overline{z}}\alpha(z)(a)=\dfrac{1}{2}\left(\dfrac{\partial\alpha}{\partial x}(x+iy)(a)+i\dfrac{\partial\alpha}{\partial y}(x+iy)(a)\right)=0
	\end{equation}
	so we know that $z\mapsto\alpha(z)(a)$ is holomorphic for any fixed $a\in A$.  We need to show that $\Phi\circ \alpha$ is holomorphic.
	Since
	\begin{equation}
		(\Phi\circ\alpha)^{\wedge}(z,a):=(\Phi\circ\alpha)(z)(a)=Exp_{\gamma_{2}(a)}^{-1}Exp_{\gamma_{1}(a)}\left(\alpha(z)(a)\right)
	\end{equation}
	is holomorphic for any fixed $a\in A$. So
	\begin{equation}
	\dfrac{1}{2}\left(\dfrac{\partial(\Phi\circ\alpha)^{\wedge}}{\partial x}(x+iy,a)+i\dfrac{\partial(\Phi\circ\alpha)^{\wedge}}{\partial y}(x+iy,a)\right)=0
	\end{equation}
	which means that
	\begin{equation}
		\dfrac{\partial(\Phi\circ\alpha)}{\partial\overline{z}}(z)(a)=\dfrac{1}{2}\left(\dfrac{\partial(\Phi\circ\alpha)}{\partial x}(x+iy)(a)+i\dfrac{\partial(\Phi\circ\alpha)}{\partial y}(x+iy)(a)\right)=0
	\end{equation}
	
	Hence $\Phi\circ\alpha$ is a holomorphic curve, which completes the proof.
\end{proof}

\end{document}